\documentclass[11pt]{amsart}
\usepackage{amsmath}
\usepackage{amsfonts}
\usepackage{amssymb}
\usepackage[all,cmtip]{xy}           

\usepackage{dsfont}

\usepackage{bbm}
\usepackage{bbding}
\usepackage{txfonts}
\usepackage[shortlabels]{enumitem}
\usepackage{ifpdf}
\ifpdf
\usepackage[colorlinks,linkcolor=black,final,backref=page,hyperindex]{hyperref}
\else
\usepackage[colorlinks,linkcolor=black,final,backref=page,hyperindex,hypertex]{hyperref}
\fi
\usepackage{amscd}
\usepackage{tikz-cd}
\usepackage[active]{srcltx}
\usepackage{quiver}
\usepackage{mathrsfs}
\usepackage{ulem}
\usepackage{graphicx}
\newcommand{\verticalsubseteq}{\mathrel{\text{\rotatebox[origin=c]{90}{$\subseteq$}}}}

\makeatletter
\newtheorem{theorem}{Theorem}[section]
\newtheorem{proposition}[theorem]{Proposition}
\newtheorem{lemma}[theorem]{Lemma}
\newtheorem{corollary}[theorem]{Corollary}
\theoremstyle{definition} 
\newtheorem{example}[theorem]{Example}
\newtheorem{definition}[theorem]{Definition}
\newtheorem{remark}[theorem]{Remark}

\newcommand{\m}{\mathbbm{m}}
\newcommand{\id}{\mathrm{id}}
\newcommand{\ho}{\mathrm{Hom}}

\newcommand{\qa}{kQ/I}

\newcommand{\rad}{\mathrm{rad}}
\newcommand{\soc}{\mathrm{soc}}

\begin{document}
	
	\title[Quasi-biserial algebras]{
		Quasi-biserial algebras, special quasi-biserial algebras and symmetric fractional Brauer graph algebras}

	\author{Bohan Xing} \address{(Bohan Xing)
		School of Mathematical Sciences,
		Laboratory of Mathematics and Complex Systems,
		Beijing Normal University,
		Beijing 100875,
		P.R.China}
	\email{bhxing@mail.bnu.edu.cn}

	\date{\today}

	\begin{abstract}
	Biserial algebras are a classical class in the representation theory of algebras, generalizing Nakayama algebras. They were further generalized by Green and Schroll to multiserial algebras, which share many structural properties with biserial algebras. Inspired by their motivation, we introduce another generalization, called quasi-biserial algebras. We show that this class retains fundamental properties of classical biserial algebras. In the symmetric special case, we establish a correspondence with labeled ribbon graphs equipped with multiplicities, providing a combinatorial model for the algebras. Furthermore, we prove that Kauer moves on these graphs, interpreted as mutations of labeled ribbon graphs, induce derived equivalences between the associated symmetric special quasi-biserial algebras.
	\end{abstract}
	
	\renewcommand{\thefootnote}{\alph{footnote}}
\setcounter{footnote}{-1} 
\renewcommand{\thefootnote}{\alph{footnote}}
\setcounter{footnote}{-1} 
	\maketitle
	

	\allowdisplaybreaks
	
	\section*{Introduction}
	
	Biserial algebras form a natural generalization of Nakayama algebras, first introduced in \cite{Fuller}. An algebra $A$ is called biserial if the radical of any indecomposable projective $A$-module is a sum of two uniserial submodules whose intersection is either zero or simple. This class includes blocks of group algebras with cyclic defect group \cite{Jan,Ku}, generalized tilted algebras of type $A_n$ \cite{AH}, and algebras whose Auslander-Reiten sequences have at most two nonprojective summands in their middle term \cite{AR}.
	
	As shown in \cite{SW}, representation-finite biserial algebras are special biserial, meaning that they are bound quiver algebras $\qa$ with at most two incoming (resp.\ outgoing) arrows per vertex in $Q$ and some `nice' relations in $I$. Special biserial algebras play a significant role in the modular representation theory of finite groups. For example, every representation-finite block and certain tame blocks (only occurs in characteristic $2$) of group algebras are special biserial \cite{DF,Jan,Rin}. Moreover in complex representation theory of the Lorentz group, the so-called Harish-Chandra modules are defined over special biserial algebras \cite{GP}.
	Since special biserial algebras are tame \cite{WW}, their indecomposable modules admit a complete combinatorial classification via string and band techniques \cite{GP,Rin}. Furthermore, \cite{Sch} demonstrates that every symmetric special biserial algebra is uniquely determined by its combinatorial data, specifically, a ribbon graph equipped with a multiplicity function (commonly referred to as a Brauer graph).
	
	In \cite{GS2}, the authors introduce a generalization of biserial algebras called multiserial algebras, where the radical of each indecomposable projective module may decompose into a sum of more than two uniserial submodules. They further extend this notion to special multiserial algebras by relaxing the quiver restrictions imposed in the special biserial case. In fact, although special multiserial algebras are wild in general, all indecomposable modules of these algebras are multiserial \cite{GS1}. Furthermore, \cite{GS2} also establishes that every symmetric special multiserial algebra is completely characterized by combinatorial data, namely, a Brauer configuration.
	
	In this paper, we introduce and study a new generalization of biserial algebras, distinct from multiserial algebras, which we term quasi-biserial algebras. Specifically, an algebra $A$ is called quasi-biserial if for every indecomposable projective $A$-module $P$, there exist a positive integer $m$ such that $P/\rad^m(P)$ is uniserial and $\rad^m(P)$ is a sum of two uniserial submodules. We further introduce a distinguished subclass of quasi-biserial algebras, called special quasi-biserial algebras, which are defined as bound quiver algebras $\qa$ satisfying specific conditions. 
	
	We further investigate the relationships between these classes of algebras in  Proposition \ref{prop:mult-cap-qbi} and Proposition \ref{prop:special-mult-cap-qbi},  whose main results can be summarized as follows:
	$$\begin{matrix}
	\textbf{\{Multiserial algebras\}}&\cap&\textbf{\{Quasi-biserial algebras\}}&=&\textbf{\{Biserial algebras\}}\\
	\verticalsubseteq&&\verticalsubseteq&&\verticalsubseteq\\
\textbf{\{Special multiserial algebras\}}&\cap&\textbf{\{Special quasi-biserial algebras\}}&=&\textbf{\{Special biserial algebras\}}
\end{matrix}$$

	We also show that special quasi-biserial algebras share properties analogous to those of special biserial algebras. For example, we have the following result.

\begin{theorem}\textnormal{(see Theorem \ref{ssqba quotient})}
Every special quasi-biserial algebra $A=\qa$ is a quotient of a symmetric special quasi-biserial algebra.
\end{theorem}
\noindent The proof differs from the biserial \cite{WW} and multiserial \cite{GS3} settings, and relies on a generalization of the classical notion of tracks for quivers with relations introduced in Definition \ref{def:track}.
	
	Furthermore, we prove that every symmetric special quasi-biserial algebra admits a symmetric fractional (abbr. sf) Brauer graph algebra structure, as defined in \cite{LL}. Consequently, such algebras are completely classified by combinatorial data, namely, a labeled ribbon graph paired with a multiplicity function (or equivalently, an sf Brauer graph).
	
	\begin{theorem}\textnormal{(see Theorem \ref{ssqb is sf-BGA})}
		Let $A=\qa$ be a finite dimensional $k$-algebra with $k$ algebraically closed. Then $A$ is a symmetric special quasi-biserial algebra if and only if it is an sf Brauer graph algebra.
	\end{theorem}
		
	We give a necessary condition for symmetric special quasi-biserial algebras to be of finite representation type, formulated in terms of their combinatorial structure, in analogy with the classical special biserial case.
		
		\begin{proposition}\textnormal{(see Corollary \ref{LBTA is RFS})}
				Let $A$ be an sf Brauer graph algebra associated with the sf Brauer graph $(\Gamma_L,\m)$, where $\Gamma$ is a ribbon graph, $L$ is a set of labeled edges in $\Gamma$ and $\m$ is a multiplicity function. If $A$ is of finite representation type, then $(\Gamma\backslash L,\m)$ is a Brauer tree.
		\end{proposition}
\noindent Unfortunately, the converse of the proposition above does not hold. We give a counterexample via the universal cover of an algebra in Example~\ref{exa:non-rep-finite-FBTA}.
		
	We also show that Kauer moves on labeled ribbon graphs induce derived equivalences between the corresponding symmetric special quasi-biserial algebras.
		
		\begin{theorem}\textnormal{(see Theorem \ref{Kauer move})}
			Let $A$ and $B$ be two sf Brauer graph algebras associated with sf Brauer graphs $(\Gamma^A_L,\m^A)$ and $(\Gamma^B_{L'},\m^B)$, respectively. Assume that the vertex sets of $\Gamma^A$ and $\Gamma^B$ coincide, and that $\m^A=\m^B$. If $\Gamma^A_L$ can be obtained from $\Gamma^B_{L'}$ via a finite sequence of Kauer moves, then $A$ and $B$ are derived equivalent.
		\end{theorem}

		We note that this class of algebras intersects with several classes studied previously. Having established their definition and basic properties, it would be of interest to explore connections between quasi-biserial algebras and more concrete classes of algebras in future work. More concretely, we make the following observations.
	
	\begin{enumerate}[(1)]
		\item As biserial algebras often arise in the study of group algebras, quasi-biserial algebras also appear in this context; see Example~\ref{exa:group-algebra}. Moreover, such quasi-biserial projective modules occur frequently in group algebras; see, for instance, the appendix of \cite{Ben}.
		
		\item Certain distinguished subclasses of special biserial algebras (e.g.\ gentle algebras and Brauer graph algebras) are known to be closed under derived equivalence (see \cite{SZ} and \cite{AZ} respectively). In contrast, their natural analogues in the setting of special multiserial algebras (such as almost gentle algebras and Brauer configuration algebras) fail to exhibit this closure property. Indeed, concrete examples demonstrate that these algebras can be naturally derived equivalent to special quasi-biserial algebras (Example \ref{exa:almost-gentle-der} and \ref{RFS}). 
		
		In particular, we also note that special multiserial algebras can be derived equivalent to considerably more complicated algebras (see~\cite[Example 4.7]{LX}). Thus, the class of special multiserial algebras derived equivalent to special quasi-biserial algebras may form only an interesting subclass of the broader derived equivalence class.
	\end{enumerate}
	
	\medskip
	
	\textbf{Outline.} \;In Section \ref{sec:pre-1}, we establish preliminary results on (special) biserial algebras, Brauer graph algebras, and (special) multiserial algebras that will be used throughout the paper. Section \ref{sec:qb-sqb-2} introduces our generalization of (special) biserial algebras to quasi-biserial and special quasi-biserial algebras, demonstrating that they retain analogous properties to their classical counterparts. In Section \ref{sec:sf-BGA}, we recall  the definition of symmetric fractional Brauer graph algebras in terms of labeled ribbon graphs and prove several properties of them. Finally, Section \ref{sec:sym-sqb-4} proves that all symmetric special quasi-biserial algebras admit an sf Brauer graph algebra structure.

	\medskip
	\section*{Funding} This research is supported by China Scholarship Council (No. 202506040127).
			\section*{Acknowledgments}
	This paper is inspired by the work of Nengqun Li and Yuming Liu in \cite{LL}. I would like to thank Nengqun Li and Yuming Liu, for many discussions and helpful comments both on mathematics and on language of this paper.

	\bigskip

	\section{Preliminaries}\label{sec:pre-1}
	
	Let $k$ be a field and $Q$ be a finite quiver, with finite vertex set $Q_0$ and finite arrow set $Q_1$. For a path $p$ in $Q$, we denote by $s(p)$ the source vertex of $p$ and by $t(p)$ its target vertex. We will write paths from right to left, for example, $p=\alpha_{n}\alpha_{n-1}\cdots\alpha_{1}$ is a path with starting arrow $\alpha_{1}$ and ending arrow $\alpha_{n}$. A path $p$ is called a cycle (equally, a closed path) if $s(p)=t(p)$.
	The length of a path $p$ will be denoted by $\ell(p)$. For two paths $p$ and $q$, we write $p\mid q$ if $p$ is a subpath of $q$.
	
	A {\it bound quiver algebra} is a $k$-algebra $A=kQ/I$ where $Q$ is a finite quiver and $I$ is an admissible ideal in $kQ$, which means there exists $m\geq 2$ such that $\langle Q_1\rangle^m\subseteq I\subseteq \langle Q_1\rangle^2.$
	By abuse of notation we always view an element in $kQ$ as an element in $\qa$ if no confusion can arise. Unless explicitly stated otherwise, all modules considered in this paper are finitely generated left $A$-modules.
	
	In this paper, we study the indecomposable projective $A$-modules $P$ via their {\it Loewy structure}, which is represented by a diagram where the $i$-th row corresponds to the simple summands of the completely reducible module $\rad^{i-1}(A)P/\rad^{i}(A)P$ with $\rad(A)$ the radical of $A$. Each number in the diagram denotes a distinct simple module in $A$. For further details, see \cite[Page 174]{Ben}.

	\subsection{Biserial and special biserial algebras}
	\
	\smallskip
	
	For a finite dimensional $k$-algebra $A$, an indecomposable module $M$ is called {\it uniserial} if there is only one composition series for $M$. We recall the definition of biserial modules and biserial algebras in \cite{Fuller}.
	
		\begin{definition}
			Let $A$ be a finite dimensional $k$-algebra.
		\begin{enumerate}[(1)]
			\item An indecomposable left (resp.\ right) $A$-module $M$ is called {\it biserial} if it is not uniserial and there exists uniserial left (resp.\ right) modules $U$ and $V$ such that $\rad(M)=U+V$ and $U\cap V$ is either zero or simple.
			
			\item The $k$-algebra $A$ is {\it biserial} if for every indecomposable projective left or right module $P$, $P$ is either uniserial or biserial.
		\end{enumerate}
	\end{definition}
	
	Biserial algebras naturally arise in the modular representation theory of finite groups (see for example in \cite{E, Ben}).	
	In \cite{SW}, the authors prove that every representation-finite biserial algebra is special biserial. This implies that such algebras admit a presentation by a quiver with particularly well-behaved relations. We recall the precise definition below.
	
	\begin{definition}
			Let $A$ be a finite dimensional $k$-algebra. $A$ is called {\it special biserial} if it is Morita equivalent to an algebra of the form $\qa$ where $kQ$ is a path algebra and $I$ is an admissible ideal such that the following properties hold
		
		\begin{enumerate}[(1)]
			\item At every vertex $i$ in $Q$, there are at most two arrows starting at $i$ and there are at most two arrows ending at $i$.
			
			\item For every arrow $\alpha$ in $Q$, there exists at most one arrow $\beta$ such that $\beta\alpha\notin I$ and there exists at most one arrow $\gamma$ such that $\alpha\gamma\notin I$.
		\end{enumerate}
	\end{definition}

	Special biserial algebras are known to be of tame representation type \cite[Corollary 2.4]{WW}. The following result, established in \cite{SW}, characterizes the relationship between biserial and special biserial algebras.
	
	\begin{theorem}\label{sb is b}
		Let $A$ be a special biserial $k$-algebra. Then $A$ is biserial.
	\end{theorem}
	
	However, the converse of Theorem \ref{sb is b} fails, as demonstrated by an example in \cite[Section 5]{FS}.
	
	\begin{example}\label{b isn sb}
		Consider the quiver $Q$ which is given by
		\[\begin{tikzcd}
			1 && 2
			\arrow["{x_1}"{description}, shift left=5, from=1-1, to=1-3]
			\arrow["{y_1}"{description}, shift left=2, from=1-1, to=1-3]
			\arrow["{x_2}"{description}, shift left=2, from=1-3, to=1-1]
			\arrow["{y_2}"{description}, shift left=5, from=1-3, to=1-1]
		\end{tikzcd}\]
		with the relations $x_1y_2=y_1x_2=x_1x_2=y_1y_2, \;y_2x_1=x_2y_1, \;x_2x_1=y_2y_1=0$. Then the corresponding bound quiver algebra $A=\qa$ is weakly symmetric and finite-dimensional.
		Indeed, $A$ is biserial but not special biserial since $x_1x_2\neq 0$, and at the same time $y_1x_2\neq 0$ in $A$.
	\end{example}
	
	There are several equivalent characterizations of symmetric algebras; see \cite[Theorem 3.1]{Ric} and \cite{Z}. For our purposes, we use the following definition.

	\begin{definition}
Let $A$ be a finite-dimensional $k$-algebra. If there exists a $k$-linear map $\psi: A \to k$ satisfying:
\begin{enumerate}
    \item \textbf{Symmetry:} $\psi(xy) = \psi(yx)$ for all $x, y \in A$;
    \item \textbf{Non-degeneracy:} the bilinear form
    \[
        \langle x, y \rangle := \psi(xy)
    \]
    is non-degenerate, i.e., if $\psi(xy) = 0$ for all $y \in A$, then $x = 0$;
\end{enumerate}
then $A$ is called a {\it symmetric algebra}, and $\psi$ is called a {\it symmetrizing form} for $A$.
\end{definition}
	
	While any finite-dimensional $k$-algebra $A$ is a quotient of some symmetric algebra (e.g., its trivial extension $T(A)$), special biserial algebras have the stronger property that they are quotients of symmetric algebras that remain special biserial.

	\begin{theorem}\textnormal{(\cite[Theorem 1.4]{WW})}\label{quotient}
		Every special biserial algebra is a quotient of a symmetric special biserial algebra.
	\end{theorem}
	
	Symmetric special biserial algebras are particularly well understood and occur frequently in the modular representation theory of finite groups. In this context, they are more commonly known as Brauer graph algebras. Although we will formally define Brauer graph algebras in Section \ref{sec:BGA}, we first state their relationship to symmetric special biserial algebras as follows.
	
	\begin{theorem}\textnormal{(\cite[Theorem 1.1]{Sch})}\label{ssbBGA}
		Let $A=\qa$ be a finite dimensional $k$-algebra with $k$ algebraically closed. Then $A$ is a symmetric special biserial algebra if and only if $\qa$ is a Brauer graph algebra.
	\end{theorem}
	
	As we will show in Section \ref{sec:BGA}, every Brauer graph algebra is completely determined by its combinatorial data, namely, its ribbon graph and multiplicity function. Consequently, Theorem \ref{ssbBGA} establishes that all symmetric special biserial algebras admit a purely combinatorial description.
	
	\subsection{Brauer graph algebras}\label{sec:BGA}
	\
	\smallskip
	
		We review the basic notions on ribbon graphs and Brauer graph algebras. For details we refer to \cite{SS} and \cite{OZ}.
		
	\subsubsection{Brauer graphs}
	\
	\smallskip
	
	Ribbon graphs combinatorially encode the structure of oriented surfaces with boundary (see for example in \cite[Section 1.1]{OZ}). A key feature of ribbon graphs is the cyclic ordering of (half-)edges at each vertex, which captures the orientation data of the underlying surface. We begin this section by recalling their formal definition.
	
	\begin{definition}\label{def:ribbon-graph}
		A {\it ribbon graph} is a tuple $\Gamma=(V,H,s,\iota,\rho)$, where
		\begin{enumerate}[(1)]
			\item $V$ is a finite set whose elements are called vertices;
			
			\item $H$ is a finite set whose elements are called half-edges;
			
			\item $s: H\rightarrow V$ is a function;
			
			\item $\iota: H\rightarrow H$ is an involution without fixed points;
			
			\item $\rho: H\rightarrow H$ is a permutation whose cycles correspond to the sets $H_v:=s^{-1}(v)$, $v\in V$.
			
		\end{enumerate}
	\end{definition}
	
	Therefore, every ribbon graph defines a graph with vertex set $V$ whose edges are the orbits of $\iota$. For instance, the edge $\bar h:=\{h, \iota(h)\}$ is incident to the vertices $s(h)$ and $s(\iota(h))$. 
	
	For a ribbon graph   $\Gamma=(V,H,s, \iota, \rho)$ we need the following notations. For each vertex $v \in V$ we denote 
	\[
	H_v=\{h \in H \mid  s(h)=v\},
	\]
	the set of half edges incident to $v$ and denote by $\mathrm{val}(v)$ the valency of the vertex $v\in V$, namely
	\[
	\mathrm{val}(v)= |H_v|.
	\]
	In particular, a loop is counted twice in $\mathrm{val}(v).$
	For each half-edge $h\in H$ we denote 
	\[
	h^+ = \rho (h) \;\; \text{and}\;\;  h^-= \rho^{-1}(h)
	\]
	the successor and predecessor of $h$ respectively, and denote by $$\bar h=\{h, \iota(h)\}$$ the edge associated to $h$ in the graph induced by $\Gamma$. The set of all edges is denoted by $E(\Gamma)$.  
	
	Unless stated otherwise, we will assume that $\Gamma$ is connected, i.e. its underlying graph is connected.

	\begin{definition}\label{BG}
		A {\it Brauer graph} is a pair $(\Gamma,\mathbbm{m})$ consisting of a ribbon graph $\Gamma =(V,H,s, \iota, \rho)$ and a function $\mathbbm{m}: V \rightarrow \mathbb{Z}_{>0}$.
	\end{definition}
	
	That is, a Brauer graph is simply a ribbon graph with  positive integers (i.e.\ multiplicities) assigned to vertices. The function $\mathbbm{m}$ in Definition \ref{BG} is referred to as the multiplicity function and its values as multiplicities. 
	In particular,  we call a given vertex $v \in V$ in a Brauer graph is {\it truncated} if $\m(v)=1 $ and $\mathrm{val}(v)=1$. Define $$H'=H\;\big\backslash\; \left(\bigcup_{\text{$v$ truncated}}H_v\right),$$
	$$E(\Gamma)'=E(\Gamma)\;\big\backslash\; \left(\bigcup_{\text{$h\in H_v$ with $v$ truncated}}\{\bar h\}\right).$$
	A Brauer graph is called a {\it Brauer tree} if its associated ribbon graph $\Gamma$ is a tree and the multiplicity function $\m$ assigns the value $1$ to all vertices, except possibly for one vertex, which is then called the {\it exceptional vertex}.

	\subsubsection{Brauer graph algebras}
	\
	\smallskip
	
	Let $(\Gamma,\mathbbm{m})$ be a Brauer graph and $k$ be a base field.  One can associate a quiver $Q=Q_\Gamma$ and an admissible ideal of relations $I=I_\Gamma$ in the path algebra $kQ$ as follows. 
	
	\begin{enumerate}[(1)]
		\item The vertices of $Q$ correspond to the edges $E(\Gamma)$ of $\Gamma$. For every half-edge $h\in H'$, which means that $s(h)$ is not truncated, there is an arrow in $Q$ $$\alpha_h:\bar{h}\rightarrow\bar{h^+}$$ from the edge $\bar h$ to the edge $\bar{h^+}$ associated to the successor $h^+$ (see Figure \ref{fig1}). That is, $\alpha_h$ may be understood as the angle around the vertex $s(h)\in V$ starting from the half-edge $h$ and ending at $h^+$, which induces a one-to-one correspondence between the set $Q_1$ of arrows and the set $H'$ of half edges.  We call the arrow $\alpha_h\in Q_1$  an arrow around the vertex $s(h)$.
		
		In particular, at each vertex in the quiver $Q_\Gamma$ there are at most two incoming arrows and at most two outgoing arrows. 
		
		Note that  we have a natural permutation 
		\[
		\sigma \colon Q_1 \to Q_1, \alpha_h\mapsto \alpha_{h^-}
		\]
		whose orbits are in bijection with the vertices which are not truncated. Hence to each arrow $\alpha= \alpha_h$ we may associate a cycle
		$$C_\alpha=\alpha\sigma(\alpha)\cdots\sigma^l(\alpha)$$
		where $l+1$ denotes the cardinality of the $\sigma$-orbit of $\alpha$ (see Figure \ref{fig1}). We call a cycle in $Q$ a {\it special cycle} if it can be expressed in the form $C_\alpha$ with some $\alpha\in Q_1$. For simplicity, we define the multiplicity of the cycle $C_{\alpha_h}$ as the multiplicity of the vertex $s(h)$ and write $$ \m(C_{\alpha_h}) = \m(s(h)),$$ with a slight abuse of notation.
		
		\begin{figure}[ht]
			\centering

		\tikzset{every picture/.style={line width=0.75pt}} 
		
		\begin{tikzpicture}[x=0.75pt,y=0.75pt,yscale=-1,xscale=1]
				
				\draw  [fill={rgb, 255:red, 0; green, 0; blue, 0 }  ,fill opacity=1 ] (315,134) .. controls (315,131.24) and (317.24,129) .. (320,129) .. controls (322.76,129) and (325,131.24) .. (325,134) .. controls (325,136.76) and (322.76,139) .. (320,139) .. controls (317.24,139) and (315,136.76) .. (315,134) -- cycle ;
				\draw [line width=1.5]    (240,134) -- (320,134) ;
				\draw [line width=1.5]    (320,134) -- (400,134) ;
				\draw [line width=1.5]    (320,134) -- (320,214) ;
				\draw [line width=1.5]    (320,54) -- (320,134) ;
				\draw  [draw opacity=0][dash pattern={on 0.84pt off 2.51pt}] (316.83,158.8) .. controls (306.04,157.44) and (297.4,149.2) .. (295.42,138.61) -- (320,134) -- cycle ; \draw  [dash pattern={on 0.84pt off 2.51pt}] (316.83,158.8) .. controls (306.04,157.44) and (297.4,149.2) .. (295.42,138.61) ;  
				\draw  [draw opacity=0] (295.71,128.04) .. controls (298.04,118.55) and (305.77,111.19) .. (315.46,109.41) -- (320,134) -- cycle ; \draw   (295.71,128.04) .. controls (298.04,118.55) and (305.77,111.19) .. (315.46,109.41) ;  
				\draw   (307.07,107.01) -- (316.68,109.19) -- (311.56,117.6) ;
				\draw  [draw opacity=0] (325.61,109.63) .. controls (334.63,111.7) and (341.78,118.64) .. (344.16,127.54) -- (320,134) -- cycle ; \draw   (325.61,109.63) .. controls (334.63,111.7) and (341.78,118.64) .. (344.16,127.54) ;  
				\draw   (347,119) -- (344.39,128.5) -- (336.22,123) ;
				\draw  [draw opacity=0] (344.35,139.67) .. controls (342.17,149.09) and (334.65,156.48) .. (325.15,158.47) -- (320,134) -- cycle ; \draw   (344.35,139.67) .. controls (342.17,149.09) and (334.65,156.48) .. (325.15,158.47) ;  
				\draw   (333.23,160.95) -- (323.47,159.63) -- (327.83,150.8) ;
				\draw  [draw opacity=0] (384.77,139.48) .. controls (381.99,172.82) and (354.05,199) .. (320,199) .. controls (284.1,199) and (255,169.9) .. (255,134) .. controls (255,98.1) and (284.1,69) .. (320,69) .. controls (354.71,69) and (383.06,96.2) .. (384.9,130.45) -- (320,134) -- cycle ; \draw   (384.77,139.48) .. controls (381.99,172.82) and (354.05,199) .. (320,199) .. controls (284.1,199) and (255,169.9) .. (255,134) .. controls (255,98.1) and (284.1,69) .. (320,69) .. controls (354.71,69) and (383.06,96.2) .. (384.9,130.45) ;  
				\draw   (388.92,122.76) -- (384.73,131.68) -- (377.61,124.87) ;
			
			\draw (315,38) node [anchor=north west][inner sep=0.75pt]   [align=left] {$h$};
			\draw (405,125) node [anchor=north west][inner sep=0.75pt]   [align=left] {$h^+$};
			\draw (222,125) node [anchor=north west][inner sep=0.75pt]   [align=left] {$h^-$};
			\draw (340,93) node [anchor=north west][inner sep=0.75pt]   [align=left] {$\alpha_h$};
			\draw (276,90) node [anchor=north west][inner sep=0.75pt]   [align=left] {$\sigma(\alpha_h)$};
			\draw (335,155) node [anchor=north west][inner sep=0.75pt]   [align=left] {$\sigma^l(\alpha_h)$};
			\draw (378,82) node [anchor=north west][inner sep=0.75pt]   [align=left] {$C_{\alpha_h}$};

		\end{tikzpicture}

			\caption{Arrows and special cycles in the quiver $Q_\Gamma$.}
			\label{fig1}
		\end{figure}
		
		\item The ideal $I_\Gamma$ is generated by the following set of relations:
			
				\begin{enumerate}[(i)]
				\item For each edge $\bar h\in E(\Gamma)'$, we have the relation $$C_\alpha^{\m(C_\alpha)}=C_\beta^{\m(C_\beta)},$$
				where $\alpha,\beta\in Q_1$ and $t(\alpha)=t(\beta) =\bar h$, that is $\alpha$ and $\beta$ end at the same edge $\bar h \in E(\Gamma)$  (see the left of Figure \ref{fig2}). Here $C_\alpha^{ \m(C_\alpha)}= \underbrace{C_\alpha C_\alpha\dotsb C_\alpha}_{ \m(C_\alpha)}$ and $ \m(C_\alpha)$ is the multiplicity of the vertex which the cycle $C_\alpha$ is around. 
				
				\item For each arrow $\alpha\in Q_1$, we impose the monomial relation $$C_{\alpha}^{\m(C_\alpha)}\alpha=0,$$
				ensuring that the quotient algebra $kQ/I$ is finite dimensional.
				
				\item For any composable arrows $\alpha$ and $\beta$ such that  $\sigma(\alpha)\neq \beta$ (see the right of Figure \ref{fig2}), we have the monomial relation $$\alpha\beta=0.$$
				This type of relation is similar to the one in gentle algebras (see e.g.\ \cite{SS}).
			\end{enumerate}
			
			\begin{figure}[ht]
				\centering

				\tikzset{every picture/.style={line width=0.75pt}} 
				
				\begin{tikzpicture}[x=0.75pt,y=0.75pt,yscale=-1,xscale=1]
					\begin{scope}[xshift=-120pt]
						
						\draw  [fill={rgb, 255:red, 0; green, 0; blue, 0 }  ,fill opacity=1 ] (140,105) .. controls (140,102.24) and (142.24,100) .. (145,100) .. controls (147.76,100) and (150,102.24) .. (150,105) .. controls (150,107.76) and (147.76,110) .. (145,110) .. controls (142.24,110) and (140,107.76) .. (140,105) -- cycle ;
						\draw  [fill={rgb, 255:red, 0; green, 0; blue, 0 }  ,fill opacity=1 ] (290,105) .. controls (290,102.24) and (292.24,100) .. (295,100) .. controls (297.76,100) and (300,102.24) .. (300,105) .. controls (300,107.76) and (297.76,110) .. (295,110) .. controls (292.24,110) and (290,107.76) .. (290,105) -- cycle ;
						\draw [line width=1.5]    (145,105) -- (295,105) ;
						\draw [line width=1.5]    (98,57.5) -- (145,105) ;
						\draw [line width=1.5]    (295,105) -- (342,152.5) ;
						\draw [line width=1.5]    (145,105) -- (98,152.5) ;
						\draw [line width=1.5]    (342,57.5) -- (295,105) ;
						\draw  [draw opacity=0] (194.27,113.55) .. controls (190.22,137.09) and (169.7,155) .. (145,155) .. controls (117.39,155) and (95,132.61) .. (95,105) .. controls (95,77.39) and (117.39,55) .. (145,55) .. controls (170.87,55) and (192.16,74.65) .. (194.74,99.84) -- (145,105) -- cycle ; \draw   (194.27,113.55) .. controls (190.22,137.09) and (169.7,155) .. (145,155) .. controls (117.39,155) and (95,132.61) .. (95,105) .. controls (95,77.39) and (117.39,55) .. (145,55) .. controls (170.87,55) and (192.16,74.65) .. (194.74,99.84) ;  
						\draw  [draw opacity=0][dash pattern={on 0.84pt off 2.51pt}] (128.61,116.47) .. controls (126.34,113.22) and (125,109.27) .. (125,105) .. controls (125,100.75) and (126.33,96.81) .. (128.59,93.57) -- (145,105) -- cycle ; \draw  [dash pattern={on 0.84pt off 2.51pt}] (128.61,116.47) .. controls (126.34,113.22) and (125,109.27) .. (125,105) .. controls (125,100.75) and (126.33,96.81) .. (128.59,93.57) ;  
						\draw  [draw opacity=0] (133.16,88.88) .. controls (136.48,86.44) and (140.57,85) .. (145,85) .. controls (154.73,85) and (162.84,91.95) .. (164.63,101.16) -- (145,105) -- cycle ; \draw   (133.16,88.88) .. controls (136.48,86.44) and (140.57,85) .. (145,85) .. controls (154.73,85) and (162.84,91.95) .. (164.63,101.16) ;  
						\draw   (168.31,93.43) -- (164.82,102.3) -- (157.29,96.46) ;
						\draw   (198.94,93.06) -- (194.34,101.4) -- (187.63,94.63) ;
						\draw  [draw opacity=0] (245.42,98.46) .. controls (248.63,73.93) and (269.6,55) .. (295,55) .. controls (322.61,55) and (345,77.39) .. (345,105) .. controls (345,132.61) and (322.61,155) .. (295,155) .. controls (269.74,155) and (248.86,136.27) .. (245.48,111.94) -- (295,105) -- cycle ; \draw   (245.42,98.46) .. controls (248.63,73.93) and (269.6,55) .. (295,55) .. controls (322.61,55) and (345,77.39) .. (345,105) .. controls (345,132.61) and (322.61,155) .. (295,155) .. controls (269.74,155) and (248.86,136.27) .. (245.48,111.94) ;  
						\draw   (240.61,119.1) -- (245.34,110.83) -- (251.94,117.71) ;
						\draw  [draw opacity=0][dash pattern={on 0.84pt off 2.51pt}] (311.19,93.26) .. controls (313.59,96.56) and (315,100.61) .. (315,105) .. controls (315,109.06) and (313.79,112.84) .. (311.71,115.99) -- (295,105) -- cycle ; \draw  [dash pattern={on 0.84pt off 2.51pt}] (311.19,93.26) .. controls (313.59,96.56) and (315,100.61) .. (315,105) .. controls (315,109.06) and (313.79,112.84) .. (311.71,115.99) ;  
						
						\draw (150,72) node [anchor=north west][inner sep=0.75pt]   [align=left] {$\alpha$};
						\draw (74,96) node [anchor=north west][inner sep=0.75pt]   [align=left] {$C_\alpha$};
						\draw (349,96) node [anchor=north west][inner sep=0.75pt]   [align=left] {$C_\beta$};
						\draw (213,85) node [anchor=north west][inner sep=0.75pt]   [align=left] {$\bar h$};
				
					\end{scope}
					\begin{scope}[xshift=120pt]
							\draw  [fill={rgb, 255:red, 0; green, 0; blue, 0 }  ,fill opacity=1 ] (140,105) .. controls (140,102.24) and (142.24,100) .. (145,100) .. controls (147.76,100) and (150,102.24) .. (150,105) .. controls (150,107.76) and (147.76,110) .. (145,110) .. controls (142.24,110) and (140,107.76) .. (140,105) -- cycle ;
							\draw  [fill={rgb, 255:red, 0; green, 0; blue, 0 }  ,fill opacity=1 ] (290,105) .. controls (290,102.24) and (292.24,100) .. (295,100) .. controls (297.76,100) and (300,102.24) .. (300,105) .. controls (300,107.76) and (297.76,110) .. (295,110) .. controls (292.24,110) and (290,107.76) .. (290,105) -- cycle ;
							\draw [line width=1.5]    (145,105) -- (295,105) ;
							\draw [line width=1.5]    (98,57.5) -- (145,105) ;
							\draw [line width=1.5]    (295,105) -- (342,152.5) ;
							\draw [line width=1.5]    (145,105) -- (98,152.5) ;
							\draw  [draw opacity=0][dash pattern={on 0.84pt off 2.51pt}] (265.53,99.33) .. controls (268.19,85.47) and (280.37,75) .. (295,75) .. controls (311.57,75) and (325,88.43) .. (325,105) .. controls (325,110.47) and (323.53,115.61) .. (320.97,120.03) -- (295,105) -- cycle ; \draw  [dash pattern={on 0.84pt off 2.51pt}] (265.53,99.33) .. controls (268.19,85.47) and (280.37,75) .. (295,75) .. controls (311.57,75) and (325,88.43) .. (325,105) .. controls (325,110.47) and (323.53,115.61) .. (320.97,120.03) ;  
							\draw  [draw opacity=0] (312.55,129.33) .. controls (307.62,132.9) and (301.55,135) .. (295,135) .. controls (279.78,135) and (267.21,123.67) .. (265.26,108.98) -- (295,105) -- cycle ; \draw   (312.55,129.33) .. controls (307.62,132.9) and (301.55,135) .. (295,135) .. controls (279.78,135) and (267.21,123.67) .. (265.26,108.98) ;  
							\draw   (259.85,117.35) -- (264.9,107.37) -- (271.95,116.06) ;
							\draw  [draw opacity=0][dash pattern={on 0.84pt off 2.51pt}] (122.42,124.76) .. controls (117.8,119.48) and (115,112.57) .. (115,105) .. controls (115,98.37) and (117.15,92.24) .. (120.8,87.27) -- (145,105) -- cycle ; \draw  [dash pattern={on 0.84pt off 2.51pt}] (122.42,124.76) .. controls (117.8,119.48) and (115,112.57) .. (115,105) .. controls (115,98.37) and (117.15,92.24) .. (120.8,87.27) ;  
							\draw  [draw opacity=0] (126.44,81.43) .. controls (131.55,77.4) and (137.99,75) .. (145,75) .. controls (159.57,75) and (171.72,85.39) .. (174.43,99.17) -- (145,105) -- cycle ; \draw   (126.44,81.43) .. controls (131.55,77.4) and (137.99,75) .. (145,75) .. controls (159.57,75) and (171.72,85.39) .. (174.43,99.17) ;  
							\draw  [draw opacity=0] (174.63,109.73) .. controls (172.36,124.05) and (159.96,135) .. (145,135) .. controls (138.84,135) and (133.12,133.15) .. (128.36,129.96) -- (145,105) -- cycle ; \draw   (174.63,109.73) .. controls (172.36,124.05) and (159.96,135) .. (145,135) .. controls (138.84,135) and (133.12,133.15) .. (128.36,129.96) ;  
							\draw   (131.27,139.67) -- (126.39,129.61) -- (137.56,129.25) ;
							\draw   (178.54,89.32) -- (174.25,99.65) -- (166.57,91.52) ;
							
							\draw (160,60) node [anchor=north west][inner sep=0.75pt]   [align=left] {$\sigma(\alpha)$};
							\draw (163,131) node [anchor=north west][inner sep=0.75pt]   [align=left] {$\alpha$};
							\draw (268,131) node [anchor=north west][inner sep=0.75pt]   [align=left] {$\beta$};
					\end{scope}
				\end{tikzpicture}
				\caption{Subgraphs corresponding to relations in $I_\Gamma$.}
				\label{fig2}
			\end{figure}
			
	\end{enumerate}
	
	\begin{definition}
		A $k$-algebra $A$ is called a {\it Brauer graph algebra} if there exists a Brauer graph $(\Gamma, \m)$ such that $A\cong kQ_\Gamma/I_\Gamma$ as $k$-algebras.
	\end{definition}
	
	The following theorem collects some general properties on Brauer graph algebras.
	
	\begin{theorem}\textnormal{(cf. \cite[Section 2.5]{SS})}\label{thm:BTA-RFS}
		Let $A=kQ/I$ be a Brauer graph algebra with its associated Brauer graph $(\Gamma,\m)$.
		\begin{enumerate}[(1)]
			\item $A$ is finite dimensional, symmetric and special biserial.
			
			\item $A$ is of finite representation type if and only if it is a Brauer tree algebra, that is, $(\Gamma,\m)$ is a Brauer tree.
		\end{enumerate}
	\end{theorem}
	
	We conclude this section with an example of a Brauer graph algebra, which will also be used later in the paper.
	
	\begin{example}\label{exa:BGA-with-4-mult-edges}
		Consider a ribbon graph $\Gamma=(V,H,s,\iota,\rho)$ which is given by
		\begin{itemize}
			\item $V=\{v_1,\;v_2\};$
			\item $H=\{1,1',2,2',3,3',4,4'\}$.
			\item Define $s(i)=v_1$, $s(i')=v_2$ for $i=1,2,3,4$.
			\item The orbits of $\iota$ (i.e.\ the edges of $\Gamma$) are given by $\{i,i'\}$ with $i=1,2,3,4$.
			\item Define $\rho(1)=2$, $\rho(2)=3$, $\rho(3)=4$, $\rho(4)=1$, $\rho(1')=2'$, $\rho(2')=3'$, $\rho(3')=4'$, $\rho(4')=1'$.
		\end{itemize}
		This ribbon graph can also be represented as follows with clockwise cyclic orientations at each vertex.
		\begin{center}

			\tikzset{every picture/.style={line width=0.75pt}} 
			
			\begin{tikzpicture}[x=0.75pt,y=0.75pt,yscale=-1,xscale=1]
				
				\draw  [fill={rgb, 255:red, 0; green, 0; blue, 0 }  ,fill opacity=1 ] (140,105) .. controls (140,102.24) and (142.24,100) .. (145,100) .. controls (147.76,100) and (150,102.24) .. (150,105) .. controls (150,107.76) and (147.76,110) .. (145,110) .. controls (142.24,110) and (140,107.76) .. (140,105) -- cycle ;
				\draw  [fill={rgb, 255:red, 0; green, 0; blue, 0 }  ,fill opacity=1 ] (290,105) .. controls (290,102.24) and (292.24,100) .. (295,100) .. controls (297.76,100) and (300,102.24) .. (300,105) .. controls (300,107.76) and (297.76,110) .. (295,110) .. controls (292.24,110) and (290,107.76) .. (290,105) -- cycle ;
				\draw [line width=1.5]    (145,105) -- (295,105) ;
				\draw [line width=1.5]    (145,105) .. controls (64,49.5) and (160,49.5) .. (295,105) ;
				\draw [line width=1.5]    (145,105) .. controls (252,49.5) and (391,46.5) .. (295,105) ;
				\draw [line width=1.5]    (145,105) .. controls (159,32.5) and (274,34.5) .. (295,105) ;
				
				\draw (165,111) node [anchor=north west][inner sep=0.75pt]   [align=left] {$1$};
				\draw (99,72) node [anchor=north west][inner sep=0.75pt]   [align=left] {$2$};
				\draw (135,72) node [anchor=north west][inner sep=0.75pt]   [align=left] {$3$};
				\draw (163,74) node [anchor=north west][inner sep=0.75pt]   [align=left] {$4$};
				\draw (263,111) node [anchor=north west][inner sep=0.75pt]   [align=left] {$1'$};
				\draw (261,76) node [anchor=north west][inner sep=0.75pt]   [align=left] {$2'$};
				\draw (291,72) node [anchor=north west][inner sep=0.75pt]   [align=left] {$3'$};
				\draw (333,72) node [anchor=north west][inner sep=0.75pt]   [align=left] {$4'$};
				\draw (124,109) node [anchor=north west][inner sep=0.75pt]   [align=left] {$v_1$};
				\draw (306,109) node [anchor=north west][inner sep=0.75pt]   [align=left] {$v_2$};

			\end{tikzpicture}
			
		\end{center}
		Define $\m(v_1)=\m(v_2)=1$. Then the quiver $Q_\Gamma$ associated with $(\Gamma,\m)$ is given by 
$$
\begin{tikzcd}
	\bar 1 \arrow[r, "\alpha_1", shift left=2] \arrow[r, "\beta_1"', shift right] & \bar 2 \arrow[d, "\alpha_2", shift left=2] \arrow[d, "\beta_2"', shift right] \\
	\bar 4 \arrow[u, "\alpha_4", shift left=2] \arrow[u, "\beta_4"', shift right] & \bar 3 \arrow[l, "\beta_3"', shift right] \arrow[l, "\alpha_3", shift left=2]
\end{tikzcd}$$
and the ideal $I_\Gamma$ is generated by the relations
\begin{enumerate}[(i)]
	\item $\alpha_4\alpha_3\alpha_2\alpha_1=\beta_4\beta_3\beta_2\beta_1$, $\alpha_3\alpha_2\alpha_1\alpha_4=\beta_3\beta_2\beta_1\beta_4$, $\alpha_2\alpha_1\alpha_4\alpha_3=\beta_2\beta_1\beta_4\beta_3$, $\alpha_1\alpha_4\alpha_3\alpha_2=\beta_1\beta_4\beta_3\beta_2$;
	\item  $\alpha_4\alpha_3\alpha_2\alpha_1\alpha_4=\alpha_3\alpha_2\alpha_1\alpha_4\alpha_3=\alpha_2\alpha_1\alpha_4\alpha_3\alpha_2=\alpha_1\alpha_4\alpha_3\alpha_2\alpha_1=0$,\\ $\beta_4\beta_3\beta_2\beta_1\beta_4=\beta_3\beta_2\beta_1\beta_4\beta_3=\beta_2\beta_1\beta_4\beta_3\beta_2=\beta_1\beta_4\beta_3\beta_2\beta_1=0$;
	\item $\beta_1\alpha_4=\alpha_1\beta_4=\beta_2\alpha_1=\alpha_2\beta_1=\beta_3\alpha_2=\alpha_3\beta_2=\beta_4\alpha_3=\alpha_4\beta_3=0$.
\end{enumerate}
The Brauer graph algebra associated with $(\Gamma,\m)$ is given by $A=kQ_\Gamma/I_\Gamma$. The indecomposable projective $A$-modules are given by	$$P_1=\begin{array}{*{3}{lll}}
	&1&\\2&&2\\3&&3\\4&&4\\&1&
\end{array},\quad P_2=\begin{array}{*{3}{lll}}
	&2&\\3&&3\\4&&4\\1&&1\\&2&
\end{array},\quad P_3=\begin{array}{*{3}{lll}}
	&3&\\4&&4\\1&&1\\2&&2\\&3&
\end{array},\quad P_4=\begin{array}{*{3}{lll}}
	&4&\\1&&1\\2&&2\\3&&3\\&4&
\end{array}.$$
Therefore, $A$ is symmetric and special biserial.		
	\end{example}
	
		\subsection{Multiserial and special multiserial algebras}
		\
		\smallskip
		
In \cite{GS2}, the notions of biserial and special biserial algebras were extended to multiserial and special multiserial algebras. We recall their definitions as follows.
		\begin{definition}
	Let $A$ be a finite dimensional $k$-algebra.
	\begin{enumerate}[(1)]
		\item An indecomposable left (resp.\ right) $A$-module $M$ is called {\it multiserial} if $\rad(M)$ can be written as a sum of uniserial left (resp.\ right) modules $U_1,\cdots, U_n$ such that, if $i\neq j$, then $U_i\cap U_j$ is either zero or simple.
		
		\item The $k$-algebra $A$ is {\it multiserial} if every indecomposable projective left or right $A$-module is multiserial.
	\end{enumerate}
\end{definition}

	\begin{definition}
	A $k$-algebra $A$ is called {\it special multiserial} if $A$ is Morita equivalent to a quotient $kQ/I$ with $I$ admissible such that for every arrow $\alpha$ in $Q$, there is at most one arrow $\beta$ in $Q$ such that $\alpha\beta\notin I$ and at most one arrow $\gamma$ in $Q$ such that $\gamma\alpha\notin I$. 
\end{definition}

In fact, special multiserial algebras share properties analogous to those of special biserial algebras. We list some key results below.

\begin{theorem}\textnormal{(\cite[Corollary 2.4]{GS1})}
	Let $A$ be a special multiserial $k$-algebra. Then $A$ is multiserial.
\end{theorem}

	\begin{theorem}\textnormal{(\cite{GS3})}
	Every special multiserial algebra is a quotient of a symmetric special multiserial algebra.
\end{theorem}

In \cite{GS1}, the authors show that symmetric special multiserial algebras also admit a combinatorial description, which is called Brauer configuration algebras. We omit a detailed discussion of Brauer configuration algebras and instead briefly state a result similar to the special biserial case.

\begin{theorem}\textnormal{(\cite[Theorem 4.1]{GS1})}
	Let $A=\qa$ be a finite-dimensional $k$-algebra with $k$ algebraically closed. Then $A$ is a symmetric special multiserial algebra if and only if $\qa$ is a Brauer configuration algebra.
\end{theorem}

While special biserial algebras are always of tame representation type, special multiserial algebras can in general be wild.

\begin{example}\label{exa:almost-gentle-der}
	Let us consider the path algebra $kQ$ where the quiver $Q$ is defined as follows.
	$$\begin{tikzcd}
	Q: & 1 & 2 \arrow[l] \arrow[r, shift left=2] \arrow[r, shift right] & 3 &  & Q': & 1 \arrow[r] & 2 \arrow[r, shift left=2] \arrow[r, shift right] & 3
\end{tikzcd}$$
Then $kQ$ is special multiserial and of wild representation type (see e.g.\ \cite{ASS}). Actually, $kQ$ is an almost gentle algebra which is defined in \cite{GS4}. According to \cite{BGP}, the path algebras $kQ$ and $kQ'$ are derived equivalent. We note that $kQ'$ is not special multiserial.
	
\end{example}
	
	\section{Quasi-biserial and special quasi-biserial algebras}\label{sec:qb-sqb-2}

	\subsection{Quasi-biserial algebras}
	\
	\smallskip
	
	In this subsection, we construct a generalization of biserial algebras, which we call quasi-biserial algebras.
	
	\begin{definition}\label{qb} Let $A$ be a finite dimensional $k$-algebra.
		\begin{enumerate}[(1)]
			\item An indecomposable left (resp.\ right) $A$-module $M$ is called {\it quasi-biserial} if it is not uniserial and there exists a positive integer $m$ and uniserial left (resp.\ right) modules $U$ and $V$, such that $\rad^m(M)=U+V$ and $M/\rad^m(M)$ is uniserial.
			
			\item The $k$-algebra $A$ is {\it quasi-biserial} if for every indecomposable projective left or right module $P$, $P$ is either uniserial or quasi-biserial.
		\end{enumerate}
	
		\end{definition}
		
			We demonstrate that quasi-biserial algebras can also arise naturally in the representation theory of finite groups, as illustrated by the following example.
		
		\begin{example}\label{exa:group-algebra}
		Consider an algebra of quaternion type (see \cite[Theorem VII.7.1]{E}) given by the quiver
			$$
\begin{tikzcd}
1 \arrow["\alpha"', loop, distance=2em, in=215, out=145] \arrow[r, "\beta", shift left] & 2 \arrow[l, "\gamma", shift left]
\end{tikzcd}
$$ with relations $$\gamma\beta\gamma=\alpha\gamma\beta\alpha\gamma,\; \beta\gamma\beta=\beta\alpha\gamma\beta\alpha,\;\alpha^2=\gamma\beta\alpha\gamma\beta,\;\beta\alpha^2=0.$$
The indecomposable projective modules are given by $$P_1=\begin{array}{*{3}{lll}}
	&1&\\&&2\\1&&1\\&&1\\&&2\\&1&\\&1&
\end{array},\quad P_2=\begin{array}{*{3}{lll}}
	&2&\\&1&\\&&1\\2&&2\\&&1\\&&1\\&2&
\end{array}.$$
It follows that this algebra is quasi-biserial but not biserial.
		\end{example}
		
		The following proposition establishes a relationship between biserial and quasi-biserial algebras.
		
		\begin{proposition}\label{prop:mult-cap-qbi}
			Let $A$ be a finite dimensional $k$-algebra. If $A$ is both multiserial and quasi-biserial, then $A$ is biserial.
		\end{proposition}
		
		\begin{proof}
			Let $P$ be an indecomposable projective $A$-module. Since $A$ is multiserial, the radical of $P$ admits a decomposition $$\rad(P)=U_1+\cdots+U_n$$ for some uniserial submodules $U_1,\cdots,U_n$. Consequently, the quotient $\rad(P)/\rad^2(P)$ decomposes into $n$ simple summands. On the other hand, the quasi-biserial property of $A$ implies that $\rad(P)/\rad^2(P)$ has at most two simple summands. Thus $n\leq 2$, providing that $P$ is biserial. It follows that $A$ itself is biserial.
		\end{proof}
		
		\subsection{Special quasi-biserial algebras}
		\
		\smallskip
		
		We now define special quasi-biserial algebras as those quasi-biserial algebras that admits a presentation as a bound quiver algebra $\qa$ where the quiver $Q$ and the admissible ideal $I$ satisfy some `nice' properties. Before this, we give some notations for quivers whose vertices have at most two incoming arrows and at most two outgoing arrows.
		
		\begin{definition}
			Let $Q$ be a finite quiver. Assume that for each vertex $i$ in $Q_0$, there are at most two arrows starting at $i$ and there are at most two arrows ending at $i$. Then for each arrow $\alpha\in Q_1$,
			\begin{itemize}
				\item we call $\alpha$ a {\it non-single arrow} if there exist another $\beta\in Q_1$ such that $s(\beta)=s(\alpha)$ or $t(\beta)=t(\alpha)$;
				
				\item we call $\alpha$ a {\it single arrow} if $\alpha$ is not a non-single arrow.
			\end{itemize}
		\end{definition}
		
		\begin{example}
			Let $Q$ be the quiver as follows.
			$$
			\begin{tikzcd}
				\bullet \arrow[r, "\alpha_1"] & \bullet \arrow[r, "\alpha_2"]                        & \bullet                        &                                                        & \bullet \arrow[ld, "\alpha_7"']                                                         &                                 \\
				&                                                      &                                & \bullet \arrow[lu, "\alpha_5"'] \arrow[rd, "\alpha_8"] &                                                                                         & \bullet \arrow[lu, "\alpha_9"'] \\
				\bullet                       & \bullet \arrow[l, "\alpha_3"'] \arrow[r, "\alpha_4"] & \bullet \arrow[ru, "\alpha_6"] &                                                        & \bullet \arrow[ru, "\alpha_{10}", shift left] \arrow[ru, "\alpha_{11}"', shift right=2] &                                
			\end{tikzcd}$$
			Then only $\alpha_1$ and $\alpha_9$ are single arrows.
		\end{example}
		
		We begin to give the precise definition of special quasi-biserial algebras.
		
		\begin{definition}\label{sqb}
	A $k$-algebra $A$ is called {\it special quasi-biserial} if it is Morita equivalent to an algebra of the form $\qa$ where $kQ$ is a path algebra and $I$ is an admissible ideal such that the following properties hold
	
	\begin{enumerate}[(1)]
		\item At every vertex $i$ in $Q$, there are at most two arrows starting at $i$ and there are at most two arrows ending at $i$.
		
		\item For each path $p$ containing at least one non-single arrow in $Q$, there exists at most one arrow $\beta$ such that $\beta p\notin I$ and there exists at most one arrow $\gamma$ such that $p\gamma\notin I$.
	\end{enumerate}
			\end{definition}
	
	Note that the path algebra $kQ'$ in Example \ref{exa:almost-gentle-der} is special quasi-biserial and of wild representation type. The following proposition establishes a relationship between special biserial and special quasi-biserial algebras.
	
	\begin{proposition}\label{prop:special-mult-cap-qbi}
		Let $A=kQ/I$ be a finite dimensional $k$-algebra. If $A$ is both special multiserial and special quasi-biserial, then $A$ is special biserial.
	\end{proposition}
	
	\begin{proof}
		Since $A$ is special quasi-biserial, for each vertex $i$ in $Q$, there are at most two arrows starting at $i$ and there are at most two arrows ending at $i$. On the other hand, the special multiserial property of $A$ providing that $A$ is special biserial.
	\end{proof}
	
	\subsection{Properties of special quasi-biserial algebras}
	\
	\smallskip
	
	In this section, we show that special quasi-biserial algebras share properties analogous to those of special biserial algebras.
	
	\begin{theorem}\label{sqb is qb}
		Let $A=\qa$ be a special quasi-biserial $k$-algebra. Then $A$ is a quasi-biserial algebra. 
	\end{theorem}
	
	\begin{proof}
		For each $i\in Q_0$, if there is a single arrow $a\in Q_1$ with $s(a)=i$, then we can find a unique path $p$ in $Q$, such that $s(p)=i$, and for each arrow $\alpha\mid p$, $\alpha$ is a single arrow in $Q$. Denote by $j=t(p)\in Q_0$ and $m=\ell(p)\in \mathbb{Z}_{>0}$. Moreover, there exists a non-single arrow $\alpha_1$ such that $s(\alpha_1)=j$. If there is no single arrow starting from $i$, denote $p=e_i$, the trivial path at $i$, and $j=i$. Without loss of generality, we assume that there exists another non-single arrow $\alpha_{2}\neq\alpha_{1}$ such that $s(\alpha_{2})=s(\alpha_{1})=j$. Otherwise, by Definition \ref{sqb}, $Ae_i$ is uniserial.
		
		We assume $\alpha_1 p$ and $\alpha_2 p$ are not zero in $A$. Therefore, by the property of special quasi-biserial algebras, there exists a unique path $w_1=p_1\alpha_1 p$ and a unique path $w_2=p_2\alpha_2 p$, such that they are maximal non zero paths starting from $i$ in $A$. Indeed, $p_1,p_2$ may be trivial paths sometimes. Moreover, let $p_1=\beta_{k}\cdots\beta_1$ and $p_2=\gamma_l\cdots\gamma_1$, then the left $A$-module $U=A(\alpha_{1}p)$ have a $k$-basis $\{\alpha_{1}p,\;\beta_1\alpha_{1}p,\;\cdots,\;\beta_{k}\cdots\beta_1\alpha_{1}p=w_1\}$ and the left $A$-module $V=A(\alpha_{2}p)$ have a $k$-basis $\{\alpha_{2}p,\;\gamma_1\alpha_{2}p,\;\cdots,\;\gamma_l\cdots\gamma_1\alpha_{2}p=w_2\}$. 
		
		It is obvious to prove that $U,V$ are uniserial and $U+V=\rad^{m+1}(Ae_i)$.
		Now consider $$Ae_i/\rad^{m+1}(Ae_i)=Ae_i/\rad (Ap).$$
		Since $p$ is a path starting at $i$ in $Q$ composed by single arrows, by Definition \ref{sqb}, we have $Ae_i/\rad^{m+1}(Ae_i)$ is also uniserial.
		
		If $\alpha_1 p=0$ or $\alpha_2 p=0$ in $A$, by Definition \ref{sqb}, $Ae_i$ is uniserial. Since the conditions of special quasi-biserial algebras are left-right symmetric, $A$ is quasi-biserial.
	\end{proof}
	
	\begin{remark}
			This theorem is not true in reverse. The biserial algebra $A=\qa$ in Example \ref{b isn sb} can be seen as a counterexample. Since $A$ is biserial, it is also quasi-biserial. However, all arrows in Q are non-single arrows, but we still have $x_1x_2\neq 0$, and at the same time $y_1x_2\neq 0$. Therefore, $A$ is not special quasi-biserial.
	\end{remark}
	
	We begin to show that each special quasi-biserial algebra can be regard as a quotient of some symmetric special quasi-biserial algebra. Here we generalized the methods used in \cite{WW} and \cite{Z}.
	
	\begin{definition}\label{def:track}
A {\it generalized track} $T=(Q,I,v)$ is a quiver $Q$ with relations $I$ and a path $v$, such that the following conditions hold:

\begin{itemize}
	\item $\qa$ is special quasi-biserial;
	
	\item all arrows in $Q_1$ occur in $v$;
	
	\item each non-single arrow $\alpha$ of $Q$ occurs precisely once in $v$ and each single arrow $\alpha$ of $Q$ occurs at most twice in $v$;
	
	\item $v$ runs at most twice through each given vertex in $Q$ (count $s(v)$ and $t(v)$ only once if $s(v)=t(v)$).
	\end{itemize}
	We say that two generalized track $T_1=(Q_1,I_1,v_1)$ and $T_2=(Q_2,I_2,v_2)$ are equivalent if $(Q_1,I_1)=(Q_2,I_2)$.
	\end{definition}
	
We show that for each special quasi-biserial algebra, we can always find a set of generalized track which is uniquely determined by its quiver and relations.
	
	\begin{lemma}\label{mark track}
		Let $A=\qa$ be an indecomposable special quasi-biserial algebra with at least one non-single arrow in its quiver $Q$. Then there is a subset of arrows $\Omega\subseteq Q_1$, such that for each $\alpha\in\Omega$, we can find a generalized track $(Q_\alpha,I\cap kQ_\alpha,v_\alpha)$ which has the following properties:
		
		\begin{itemize}
			\item each $Q_\alpha$ is a subquiver of $Q$;
			
			\item each non-single arrow of $Q$ belongs to exactly one $Q_\alpha$;
			
			\item each single arrow $\beta$ of $Q$ must belong to some $Q_\alpha$. In particular, if there exists a path $p$ containing non-single arrows in $Q$, such that $p\alpha\notin I$ or $\alpha p\notin I$, then $\beta$ belongs to at least two $Q_\alpha$.
		\end{itemize}  
		Moreover, such a set of generalized track is uniquely determined by the quiver and relations $(Q,I)$ of the special quasi-biserial algebra $A=\qa$.
	\end{lemma}
	
	\begin{proof}
		By Definition \ref{sqb}, for each $\alpha\in Q_1$ is a non-single arrow, there exists a unique maximal path $p\alpha\notin I$ and a unique maximal path $\alpha q\notin I$.
		Let $v_\alpha$ be the maximal subpath of $p\alpha q$ in $kQ$ such that each non-single arrow in $p\alpha q$ occurs precisely once in $v_\alpha$ and each single arrow in $p_\alpha$ occurs at most twice in $v_\alpha$ such that $v_\alpha$ fits the conditions in the definition of generalized tracks. Let $Q_\alpha$ be the underlying quiver of $v_\alpha$. Then $T_\alpha=(Q_\alpha,I\cap kQ_\alpha,v_\alpha)$ is a generalized track.
		We may define an equivalence relation on the set of non-single arrows in $Q_1$ by saying that two non-single arrows $\alpha$ and $\beta$ are equivalent if $(Q_\alpha,I\cap kQ_\alpha)=(Q_\beta,I\cap kQ_\beta)$. An equivalence class is then the set of arrows in a track. Let $\Omega'$ be a set of representatives of the classes of non-single arrows under this equivalence relation.
		
		Now consider the set $S$ of the arrows in $Q$ which are not involved in any $Q_\alpha$ above. They must be a subset of single arrows in $Q$. By Definition \ref{sqb} and the discussion above, for all $\alpha\in S$, we can find a unique  maximal path $p_\alpha$ in $A$ such that no subpath of $p_\alpha$ is in $I$. $p_\alpha$ is consist of single arrows. Let $v_\alpha$ be the maximal subpath of $p_\alpha$ such that each arrow which is involved in $p_\alpha$ occurs exactly once in $v_\alpha$. Then $T_\alpha=(Q_\alpha,I\cap kQ_\alpha,v_\alpha)$ is a generalized track. We can also define an equivalence relation on the set $S$. Let $\Omega''$ be a set of representatives of the classes of this equivalence relation. Then let $\Omega=\Omega'\cup\Omega''$.
		
		The construction of the tracks $T_\alpha$ above implies that they are unique up to equivalence relations. Then the result follows.
	\end{proof}
	
	To enhance understanding of Lemma \ref{mark track}, we provide a concrete example as follows.
	
	\begin{example}\label{exa:tracks}
		Let $Q$ be the quiver which is given by 
		$$
	\begin{tikzcd}
		1 \arrow[r, "\alpha", shift left=2] \arrow[r, "\beta"', shift right] & 2 \arrow[d, "\gamma"]                                                   \\
		4                                                                    & 3 \arrow[l, "\delta"', shift right] \arrow[l, "\epsilon", shift left=2]
	\end{tikzcd}$$
		and $I=\langle \epsilon\gamma, \delta\gamma\beta,  \delta\gamma\alpha\rangle$. Then $A=kQ/I$ is special quasi-biserial. Then $\Omega=\{\alpha,\beta,\delta,\epsilon\}$ and the set of generalized tracks in Lemma \ref{mark track} is given by
		$$\{(Q_\alpha,0,\gamma\alpha),\;(Q_\beta,0 ,\gamma\beta),\;(Q_\delta,0,\delta\gamma),\;(Q_\epsilon,0,\epsilon)\}.$$
	\end{example}
	
	Finally, we prove the main theorem of this section.
	
	\begin{theorem}\label{ssqba quotient}
		Every indecomposable special quasi-biserial algebra $A=\qa$ is a quotient of a symmetric special quasi-biserial algebra.
	\end{theorem}
	
	\begin{proof}
		If $Q$ only contains single arrows, then $A$ is simply a Nakayama algebra. Then by Theorem \ref{quotient}, $A$ is a quotient of some symmetric special biserial algebra. Now assume that $Q$ contains at least one non-single arrow.
		By Lemma \ref{mark track}, since $A$ is special quasi-biserial, we can find a unique set $\{T_\alpha\mid\alpha\in\Omega_A\}$ of generalized tracks. Let $\Sigma_A\subseteq\Omega_A$ be the set of all those arrows $\alpha$ such that $v_\alpha$ is a closed path for each $\Sigma_A$. We shall recursively construct,  by induction on $\Omega_A\backslash\Sigma_A$, a special quasi-biserial algebra $A'$ which maps onto $A$ and for which $\Sigma_{A'}=\Omega_{A'}$.

		Let $v_\alpha,v_\beta$ be paths in generalized tracks with $\alpha,\beta\in\Omega_A\backslash\Sigma_A$. Define a new operation between $v_\alpha=\alpha_{n}\cdots\alpha_{1}$ and $v_\beta=\beta_m\cdots\beta_1$.
			$$
		v_\alpha\star v_\beta:=\left\{
		\begin{array}{*{3}{lll}}
			\alpha_{n}\cdots\alpha_{l+1}\beta_{m}\cdots\beta_1& ,& \text{if $\alpha_i=\beta_{m-l+i}$ are single arrows in $Q$ for all $1\leq i\leq l$};\\
			v_\alpha v_\beta&,&\text{else, if $s(v_\alpha)=t(v_\beta)$;}\\
			0&,& \text{otherwise}.
		\end{array}
		\right.
		$$
		By definition of the star product above and the maximality of $v_\alpha,v_\beta$, we can choose a minimal generating set of $I$ containing a set $R(\alpha,\beta)$ of relations which can be divided by $\alpha_{l+1}\beta_{m}\cdots\beta_{m-l}$. If there exist some $\alpha_{j}',\beta_{k}'$ are non-single arrows in $Q$, such that $s(\alpha_{j}')=s(\alpha_{j})$ and $t(\beta_{k}')=t(\beta_{k})$ with $l+1\leq j\leq n$ and $1\leq k\leq m-l$. then we denote the set of elements of form $\alpha_{j}'\alpha_{j-1}\cdots\alpha_{l+1}\beta_{m}\cdots\beta_{k+1}\beta_{k}'$ by $p(\alpha,\beta)$.
		
			\smallskip
		
		\textit{Case 1}. Suppose that there is a subset $\Delta_A\subset\Omega_{A}\backslash\Sigma_A$ such that $\Delta_A=\{\alpha_{1},\cdots,\alpha_d\}$ and such that the paths $\{v_\alpha\mid\alpha\in\Delta_A\}$ can be ordered so that $v_{\alpha_{1}}\star\cdots\star v_{\alpha_d}$ is a closed path. 
		
		By the maximality of $v_\alpha$, we can choose a minimal generating set $\mathcal{G}$ of $I$ contains $\bigcup_{i=1}^{d-1}R(\alpha_{i},\alpha_{i+1})$. Let $$I':=\langle(\mathcal{G}\;\backslash\;\bigcup_{i=1}^{d-1}R(\alpha_{i},\alpha_{i+1}))\cup\{p(\alpha_{i},\alpha_{i+1})\mid i=1,\cdots,d-1\}\rangle.$$
		By condition (2) in Definition \ref{sqb}, $I'\subseteq I$. Therefore, $\qa$ is a quotient algebra of $\qa'$, and it is clear that $|\Omega_{A}\backslash\Sigma_A|>|\Omega_{A'}\backslash\Sigma_{A'}|$.
		
			\smallskip
		
		\textit{Case 2}. Suppose there is a maximal subset $\Delta_A\subset\Omega_{A}\backslash\Sigma_A$ such that $\Delta_A=\{\alpha_{1},\cdots,\alpha_d\}$ and such that the paths $\{v_\alpha\mid\alpha\in\Delta_A\}$ can be ordered so that $v_{\alpha_{1}}\star\cdots\star v_{\alpha_d}$ is a nonzero path which is not closed. 
		
		Since $\Delta_A$ is maximal, and hence there is no path in the generalized tracks corresponding to the elements in $\Omega_{A}\backslash(\Sigma_A\cup\Delta_A)$ is composable (by star product) with $v_{\alpha_{1}}\star\cdots\star v_{\alpha_d}$. This implies that no arrow of $Q$ starts at the terminus vertex $t(v_{\alpha_{1}})$ of $v_{\alpha_{1}}$ and no arrow of $Q$ ends at the source vertex $s(v_{\alpha_{d}})$ of $v_{\alpha_{d}}$. We add a new arrow $\sigma:t(v_{\alpha_{1}})\rightarrow s(v_{\alpha_{d}})$ to $Q$. The new quiver is named $Q'$ and the ideal $I'$ is generated by 
		$$(\mathcal{G}\;\backslash\;\bigcup_{i=1}^{d-1}R(\alpha_{i},\alpha_{i+1}))\cup\{p(\alpha_{i},\alpha_{i+1})\mid i=1,\cdots,d-1\}\cup\{\sigma\delta,\epsilon\sigma\} $$
		with $\delta$ the arrow which is different from the ending arrow of $v_{\alpha_{1}}$ and $\epsilon$ the arrow which is different from the starting arrow of $v_{\alpha_{d}}$. We observe that $A'=\qa'$ maps surjectively onto $A=\qa$. Moreover, $A'$ is special quasi-biserial and $|\Omega_{A}\backslash\Sigma_A|>|\Omega_{A'}\backslash\Sigma_{A'}|$.
		
			\smallskip
		
		Hence, by induction we may assume that all path $v_\alpha$ in the generalized tracks of $A$ are closed path (unique up to cyclic permutation). Since $A$ is finite-dimensional, let $n$ be the smallest integer such that $\rad^n(A)=0$, and for all non-single arrow $\beta\in Q_1$, there is a unique $v_\alpha$ in tracks of $A$ containing $\beta$, so we denote it by $v_\beta$. Let $\beta$ be a single arrow in $Q$. There exists at most two closed paths $v_{\alpha_{1}}$ and $v_{\alpha_{2}}$ containing $\beta$. Let $I_s$ be the ideal of $kQ$ generated by all elements
		
		\begin{enumerate}
			\item $p=\alpha_{n}\cdots\alpha_{1}$ with $\alpha_n$ and $\alpha_1$ are non-single arrows in different $v_\alpha$, while $\alpha_{n-1},\cdots,\alpha_2$ are all single arrows;
			
			\item  $\alpha_{2}\alpha_{1}$  is not a subpath of any $v_\alpha$ except if $\alpha_{2}=\alpha_1$ is a loop;
			
			\item all paths containing a proper subpath $v_\alpha^n$ (up to cyclic permutation) for some $\alpha\in\Omega_{A}$;
			
			\item $p_1-p_2$, where there exist $q_1,q_2\in Q$ of maximal length, such that $q_1(p_1-p_2)q_2=v_{\alpha}^n-v_{\beta}^n$ with $v_{\alpha}\neq v_{\beta}$, and $v_{\alpha}$ and $v_{\beta}$ starting at same vertex in $Q$. 
		\end{enumerate}
		Then $(Q,I_s)$ defines a special quasi-biserial algebra $A_s=\qa_s$ and $A$ is an epimorphic image of $A_s$. 
		
		Let us fix a $k$-basis $\mathcal{B}$ of $A_s$ consisting of pairwise distinct nonzero paths of quiver $Q$. Define a $k$-linear map $\psi$ on the basis elements in $\mathcal{B}$ by
		$$
		\psi(b):=\left\{
		\begin{array}{*{3}{lll}}
		1& ,& \text{if}\; b\in\soc(A_s);\\
		0&,& \text{otherwise.}
		\end{array}
		\right.
		$$
		Then it defines a non-degenerate linear form on $A_s$. The form is symmetric since if $uv$ is a path in the socle, then $uv=v_\alpha^n$ for some $\alpha$, and then $vu=v_\beta^n$ for some $\beta$. Moreover, if there exist a left or right ideal $J$, such that $\psi(J)=0$, then by the definition of $\psi$, there exist $v_1^n,v_2^n\in \soc(A_s)$, such that $c_1v_1^n-c_2v_2^n\in J\cap\soc(A_s)$ with $c_1,c_2\in k$. Since $\psi(J)=0$, $c_1=c_2$. However, by the relations given by (4), we have all the elements which can generate $v_1^n-v_2^n$ are zero in the algebra $A_s$. That means $J=0$. Consequently, $A_s$ is a symmetric special quasi-biserial algebra equipped with a surjective homomorphism onto $A$.
	\end{proof}
	
	As a final illustration, we apply Example \ref{exa:tracks} to showcase Theorem \ref{ssqba quotient}.
	
	\begin{example}\label{exa:As-surj-A} (Example \ref{exa:tracks} revisit)\;\;
	Recall that the set of generalized tracks of $A$ in Lemma \ref{mark track} is given by
		$$\{(Q_\alpha,0,\gamma\alpha),\;(Q_\beta,0 ,\gamma\beta),\;(Q_\delta,0,\delta\gamma),\;(Q_\epsilon,0,\epsilon)\}.$$
		Then we can choose $\delta\gamma\star\gamma\alpha=\delta\gamma\alpha$ and $\epsilon\star\gamma\beta=\epsilon\gamma\beta$ to be the maximal paths in $Q$. Then the $A_s=kQ'/I'$ in Theorem \ref{ssqba quotient} can be chosen by the following quiver $Q'$
		$$
		\begin{tikzcd}
			1 \arrow[r, "\alpha", shift left=2] \arrow[r, "\beta"', shift right]      & 2 \arrow[d, "\gamma"]                                                   \\
			4 \arrow[u, "\kappa_1", shift left=2] \arrow[u, "\kappa_2"', shift right] & 3 \arrow[l, "\delta"', shift right] \arrow[l, "\epsilon", shift left=2]
		\end{tikzcd}$$
		and relations $I'$
		\begin{enumerate}[(i)]
			\item $\alpha\kappa_1\delta=\beta\kappa_2\epsilon$, $\kappa_1\delta\gamma\alpha=\kappa_2\epsilon\gamma\beta$, $\delta\gamma\alpha\kappa_1=\epsilon\gamma\beta\kappa_2$;
			\item  $\kappa_1\delta\gamma\alpha\kappa_1=\alpha\kappa_1\delta\gamma\alpha=\gamma\alpha\kappa_1\delta\gamma=\delta\gamma\alpha\kappa_1\delta=0$,\\ $\kappa_2\epsilon\gamma\beta\kappa_2=\beta\kappa_2\epsilon\gamma\beta=\gamma\beta\kappa_2\epsilon\gamma=\epsilon\gamma\beta\kappa_2\epsilon=0$;
			\item $\alpha\kappa_2=\beta\kappa_1=\kappa_2\delta=\kappa_1\epsilon=\delta\gamma\beta=\epsilon\gamma\alpha=0$.
		\end{enumerate}
		The indecomposable projective $A_s$-modules are given by	$$P_1=\begin{array}{*{3}{lll}}
			&1&\\2&&2\\3&&3\\4&&4\\&1&
		\end{array},\quad P_2=\begin{array}{*{3}{lll}}
			&2&\\&3&\\4&&4\\1&&1\\&2&
		\end{array},\quad P_3=\begin{array}{*{3}{lll}}
			&3&\\4&&4\\1&&1\\&2&\\&3&
		\end{array},\quad P_4=\begin{array}{*{3}{lll}}
			&4&\\1&&1\\2&&2\\3&&3\\&4&
		\end{array}.$$
		Therefore, $A_s$ is naturally symmetric and special quasi-biserial, and $$A\cong A_s/\langle\kappa_1,\kappa_2,  \epsilon\gamma,  \delta\gamma\alpha\rangle.$$
		In fact, $A_s$ is a symmetric fractional Brauer graph algebra (to be defined formally in Section \ref{sec:sf-BGA}).
	\end{example}

	\section{Symmetric fractional Brauer graph algebras}\label{sec:sf-BGA}
	
	In this section, we examine a special class of symmetric special quasi-biserial algebras introduced in \cite{LL}, called symmetric fractional (abbr. sf) Brauer graph algebras. These algebras are completely determined by combinatorial data consisting of labeled ribbon graphs and multiplicity functions.
	
	\subsection{Sf Brauer graphs and their associated algebras}
	\
	\smallskip
	
	Recall that a ribbon graph is a combinatorial structure $\Gamma=(V,H,s,\iota,\rho)$, where $V$ is a vertex set, $H$ a half-edge set, $s$ assigns half-edges to their incident vertices, $\iota$ partitions $H$ into edges $E(\Gamma)$ via its orbits, and $\rho$ defines a cyclic ordering of half-edges incident to each vertex. 
	
	\begin{definition}
		Let $\Gamma=(V,H,s,\iota,\rho)$ be a ribbon graph. An edge $\bar h=\{h,\iota (h)\}\in E(\Gamma)$ is called {\it labelable} if $\iota(\rho(h))=\rho(\iota(h))$.
	\end{definition}
	
	In Example \ref{exa:BGA-with-4-mult-edges}, all edges in $\Gamma$ are labelable.
	
	\begin{definition}
		Let $\Gamma$ be a ribbon graph and $L\subseteq E(\Gamma)$ a set of labelable edges. The pair $\Gamma_L=(\Gamma, L)$ is called a {\it labeled ribbon graph} if for every $v\in V$ with $\mathrm{val}(v)\geq 2$, there exist at least two half-edges $h_i,h_j\in H_v$ such that their associated edges $\bar h_i, \bar h_j$ are not in $L$. In particular, edges in $L$ are referred to as {\it labeled edges}.
	\end{definition}
	
	In fact, labeled edges in the ribbon graph is corresponding to the $L$-partition in \cite{LL}. For simplicity, we draw these edges as dotted lines in the ribbon graph $\Gamma$. We give the definition of sf Brauer graph as follows.
	
	\begin{definition}
		An {\it sf Brauer graph} is a pair $(\Gamma_L,\mathbbm{m})$ consisting of a labeled ribbon graph $\Gamma_L=(\Gamma,L)$ and a multiplicity function $\mathbbm{m}: V \rightarrow \mathbb{Z}_{>0}$.
	\end{definition}
	
		Let $(\Gamma_L,\mathbbm{m})$ be an sf Brauer graph and $k$ be a base field.  One can associate a quiver $Q=Q_{\Gamma,L}$ and an admissible ideal of relations $I=I_{\Gamma,L}$ in the path algebra $kQ$ as follows. 
	
	\begin{enumerate}[(1)]
		\item Denote by $Q_\Gamma$, the quiver associated with the Brauer graph $(\Gamma,\m)$.  We associate to $(\Gamma_L,\mathbbm{m})$ the quiver $Q_{\Gamma,L}$ defined as the quiver of the quotient algebra $kQ_\Gamma/I_L$, where $I_L$ is the ideal of $kQ_\Gamma$ generated by $$\{\alpha-\beta\;|\;\text{$\alpha$ and $\beta$ are arrows starting from the same labeled edge in $\Gamma_L$}\}.$$ 
		The quotient identifies all arrows starting from the same labeled edge in $\Gamma_L$ as a single arrow in $Q_{\Gamma,L}$. Special cycles in $Q_{\Gamma,L}$ can be studied via the canonical projection $\pi\colon\;kQ_\Gamma\rightarrow kQ_{\Gamma,L}$ where each special cycle in $Q_{\Gamma,L}$ lifts to a collection of special cycles in $Q_\Gamma$.
		
		\item The ideal $I_{\Gamma,L}$ is generated by the following set of relations:
		
		\begin{enumerate}[(i)]
			\item For each edge $\bar h\in E(\Gamma)'$ (i.e.\ both $s(h)$ and $s(\iota(h))$ are not truncated), there are exactly two special cycles $C_\alpha$ and $C_\beta$ end at $\bar h$ in $Q_\Gamma$. 
			
			Let $\pi(C_\alpha^{\m(C_\alpha)})=\alpha_n\cdots\alpha_1$ and $\pi(C_\beta^{\m(C_\beta)})=\beta_m\cdots\beta_1$ in $Q_{\Gamma,L}$. Denote by $l_1$ (resp. $l_2$) the smallest positive integer such that $\alpha_{l_1}\neq\beta_{l_1}$ (resp. $\alpha_{n-l_2+1}\neq\beta_{m-l_2+1}$) in $Q_{\Gamma,L}$. Then we have the relation
			$$\alpha_{n-l_2+1}\cdots\alpha_{l_1}=\beta_{m-l_2+1}\cdots\beta_{l_1}.$$
			
			\item For each half-edge $h\in H'$ (i.e.\ both $s(h)$ is not truncated), we impose the monomial relation $$\pi(C_{\alpha_h}^{\m(C_{\alpha_h})}\alpha_h)=0,$$
			ensuring that the quotient algebra $kQ/I$ is finite dimensional.
			
			\item For some path $p=\alpha_n\cdots\alpha_1$ in $Q_\Gamma$, assume
			\begin{itemize}
				\item $\alpha_2,\cdots,\alpha_{n-1}$ are arrows starting from labeled edges in $\Gamma_L$;
				\item $\alpha_1$ and $\alpha_n$ are not starting from labeled edges in $\Gamma_L$;
				\item $\sigma^{n-1}(\alpha_n)\neq \alpha_1$ in $Q_\Gamma$.
			\end{itemize} Then we have the monomial relation $$\pi(p)=0.$$
		\end{enumerate}
	\end{enumerate}
	
	Same as the Brauer graph case, we give the definition of sf Brauer graph algebras.
	
	\begin{definition}
		A $k$-algebra $A$ is called an {\it sf Brauer graph algebra} if there exists an sf Brauer graph $(\Gamma_L, \m)$ such that $A\cong kQ_{\Gamma,L}/I_{\Gamma,L}$ as $k$-algebras.
	\end{definition} 
	
	We give an example here.
	
	\begin{example}\label{example sf-BGA}
		Let $\Gamma$ be the ribbon graph in Example \ref{exa:BGA-with-4-mult-edges}. Each edge in $\Gamma$ is labelable. Define $L=\{\bar 2, \bar 4\}$. Then the corresponding labeled ribbon graph can be described as follows.
		
			\begin{center}

			\tikzset{every picture/.style={line width=0.75pt}} 
			
			\begin{tikzpicture}[x=0.75pt,y=0.75pt,yscale=-1,xscale=1]
				
				\draw  [fill={rgb, 255:red, 0; green, 0; blue, 0 }  ,fill opacity=1 ] (140,105) .. controls (140,102.24) and (142.24,100) .. (145,100) .. controls (147.76,100) and (150,102.24) .. (150,105) .. controls (150,107.76) and (147.76,110) .. (145,110) .. controls (142.24,110) and (140,107.76) .. (140,105) -- cycle ;
				\draw  [fill={rgb, 255:red, 0; green, 0; blue, 0 }  ,fill opacity=1 ] (290,105) .. controls (290,102.24) and (292.24,100) .. (295,100) .. controls (297.76,100) and (300,102.24) .. (300,105) .. controls (300,107.76) and (297.76,110) .. (295,110) .. controls (292.24,110) and (290,107.76) .. (290,105) -- cycle ;
				\draw [line width=1.5]    (145,105) -- (295,105) ;
				\draw [line width=1.5]  [dash pattern={on 5.63pt off 4.5pt}]  (145,105) .. controls (64,49.5) and (160,49.5) .. (295,105) ;
				\draw [line width=1.5]  [dash pattern={on 5.63pt off 4.5pt}]  (145,105) .. controls (252,49.5) and (391,46.5) .. (295,105) ;
				\draw [line width=1.5]    (145,105) .. controls (159,32.5) and (274,34.5) .. (295,105) ;
				\draw (165,111) node [anchor=north west][inner sep=0.75pt]   [align=left] {$1$};
				\draw (99,72) node [anchor=north west][inner sep=0.75pt]   [align=left] {$2$};
				\draw (135,72) node [anchor=north west][inner sep=0.75pt]   [align=left] {$3$};
				\draw (163,74) node [anchor=north west][inner sep=0.75pt]   [align=left] {$4$};
				\draw (263,111) node [anchor=north west][inner sep=0.75pt]   [align=left] {$1'$};
				\draw (261,76) node [anchor=north west][inner sep=0.75pt]   [align=left] {$2'$};
				\draw (291,72) node [anchor=north west][inner sep=0.75pt]   [align=left] {$3'$};
				\draw (333,72) node [anchor=north west][inner sep=0.75pt]   [align=left] {$4'$};
				\draw (124,109) node [anchor=north west][inner sep=0.75pt]   [align=left] {$v_1$};
				\draw (306,109) node [anchor=north west][inner sep=0.75pt]   [align=left] {$v_2$};

			\end{tikzpicture}
			
		\end{center}
		Let $\m(v_1)=\m(v_2)=1$. Then the quiver $Q_{\Gamma,L}$ associated with $\Gamma_L$ is given by 
		$$
		\begin{tikzcd}
			1 \arrow[r, "\alpha_1", shift left=2] \arrow[r, "\beta_1"', shift right] & 2 \arrow[d, "\alpha_2"]                                                  \\
			4 \arrow[u, "\alpha_4"]                                                  & 3 \arrow[l, "\beta_3"', shift right] \arrow[l, "\alpha_3", shift left=2]
		\end{tikzcd}$$
		and the corresponding sf Brauer graph algebra is $A=kQ_{\Gamma,L}/I_{\Gamma,L}$ where $I_{\Gamma,L}$ is generated by the following relations.	
		\begin{enumerate}[(i)]
			\item $\alpha_1\alpha_4\alpha_3=\beta_1\alpha_4\beta_3$, $\alpha_3\alpha_2\alpha_1=\beta_3\alpha_2\beta_1$;
			\item $\alpha_4\alpha_3\alpha_2\alpha_1\alpha_4=\alpha_3\alpha_2\alpha_1\alpha_4\alpha_3=\alpha_2\alpha_1\alpha_4\alpha_3\alpha_2=\alpha_1\alpha_4\alpha_3\alpha_2\alpha_1=0$,\\ $\alpha_4\beta_3\alpha_2\beta_1\alpha_4=\beta_3\alpha_2\beta_1\alpha_4\beta_3=\alpha_2\beta_1\alpha_4\beta_3\alpha_2=\beta_1\alpha_4\beta_3\alpha_2\beta_1=0$;
			\item $\beta_3\alpha_2\alpha_1=\alpha_3\alpha_2\beta_1=\beta_1\alpha_4\alpha_3=\alpha_1\alpha_4\beta_3=0$.
		\end{enumerate}
		The indecomposable projective module of $A$ are 
		$$P_1=\begin{array}{*{3}{lll}}
			&1&\\2&&2\\3&&3\\&4&\\&1&
		\end{array},\quad P_2=\begin{array}{*{3}{lll}}
			&2&\\&3&\\4&&4\\1&&1\\&2&
		\end{array},\quad P_3=\begin{array}{*{3}{lll}}
			&3&\\4&&4\\1&&1\\&2&\\&3&
		\end{array},\quad P_4=\begin{array}{*{3}{lll}}
			&4&\\&1&\\2&&2\\3&&3\\&4&
		\end{array}.$$
		Therefore, $A$ is naturally symmetric and special quasi-biserial.
		
		Note that the algebra $A_s$ in Example \ref{exa:As-surj-A} is also an sf Brauer graph algebra associated with the ribbon graph $\Gamma$ in Example \ref{exa:BGA-with-4-mult-edges} and a labeled edge $\bar 2$.
	
	\end{example}
	
	We show that the following proposition hold for each sf Brauer graph algebra.
	
	\begin{proposition}\label{sf-BGA is ssqb}
		Let $A=\qa$ be an sf Brauer graph algebra associated with an sf Brauer graph $(\Gamma_L,\m)$. Then $A$ is finite dimensional, symmetric and special quasi-biserial.
	\end{proposition}
	
	\begin{proof}
		The relation (ii) in $I$ ensures that $A$ is finite dimensional. Now we fix a $k$-basis $\mathcal{B}$ of $A$ consisting of pairwise distinct nonzero paths of quiver $Q$. Assume that $\mathcal{B}$ is a $k$-linear basis of $A$ that is consist of paths. Define a $k$-linear map $\psi$ on the basis elements $\mathcal{B}$ by
		$$
		\psi(b):=\left\{
		\begin{array}{*{3}{lll}}
			1& ,& \text{if}\; b\in\soc(A);\\
			0&,& \text{otherwise.}
		\end{array}
		\right.
		$$
		 Then it defines a non-degenerate linear form on $A$. The form is symmetric since if $uv$ is a path in the socle, then $uv=\pi(C_\alpha^{\m(C_\alpha)})$ for some special cycle $C_\alpha$ in $\Gamma$, and then $vu=\pi(C_\beta^{\m(C_\beta)})$ for some special cycle $C_\beta$ in $\Gamma$. Therefore, it is in $\soc(A)$. Moreover, if there exist a left or right ideal $J$, such that $\psi(J)=0$, then by the definition of $\psi$, $J\cap\soc(A)=0$, which means $J=0$. Therefore, $A$ is symmetric.
		
		Indeed, we can regard the quiver $Q$ as a subquiver of the quiver $Q_\Gamma$. Therefore, since every vertex $i$ in $Q_\Gamma$, there are at most two arrows starting at $i$ and there are at most two arrows ending at $i$, the same condition (1) in Definition \ref{sqb} also holds for $Q$. The relations (iii) ensure the condition (2) in Definition \ref{sqb} holds. Therefore, $A$ is special quasi-biserial.
	\end{proof}
	
	\begin{remark}
		Indeed, if a special quasi-biserial algebra is an sf Brauer graph algebra, then each single arrow in its quiver is starting from either a labeled edge $\bar h\in L$ or some $\bar h\in E(\Gamma)'$, an edge connected to a truncated vertex.
	\end{remark}
	
\subsection{Representation-finite sf Brauer graph algebras}
\
\smallskip

	Since extending the scope of Brauer graph algebras to sf Brauer graph algebras, we have found that the class of algebras of finite representation type has significantly expanded.
	
	\begin{example}\label{RFS}
		Let $\Gamma_L$ be an sf Brauer graph with multiplicity $1$ at each vertex which is represented as follows (with clockwise cyclic orientations at each vertex).
		
		\begin{center}

			\tikzset{every picture/.style={line width=0.75pt}} 
			
			\begin{tikzpicture}[x=0.75pt,y=0.75pt,yscale=-1,xscale=1]
					
					\draw  [fill={rgb, 255:red, 0; green, 0; blue, 0 }  ,fill opacity=1 ] (200,126) .. controls (200,123.24) and (202.24,121) .. (205,121) .. controls (207.76,121) and (210,123.24) .. (210,126) .. controls (210,128.76) and (207.76,131) .. (205,131) .. controls (202.24,131) and (200,128.76) .. (200,126) -- cycle ;
					\draw  [fill={rgb, 255:red, 0; green, 0; blue, 0 }  ,fill opacity=1 ] (320,126) .. controls (320,123.24) and (322.24,121) .. (325,121) .. controls (327.76,121) and (330,123.24) .. (330,126) .. controls (330,128.76) and (327.76,131) .. (325,131) .. controls (322.24,131) and (320,128.76) .. (320,126) -- cycle ;
					\draw  [fill={rgb, 255:red, 0; green, 0; blue, 0 }  ,fill opacity=1 ] (250,126) .. controls (250,123.24) and (252.24,121) .. (255,121) .. controls (257.76,121) and (260,123.24) .. (260,126) .. controls (260,128.76) and (257.76,131) .. (255,131) .. controls (252.24,131) and (250,128.76) .. (250,126) -- cycle ;
					\draw  [fill={rgb, 255:red, 0; green, 0; blue, 0 }  ,fill opacity=1 ] (80,126) .. controls (80,123.24) and (82.24,121) .. (85,121) .. controls (87.76,121) and (90,123.24) .. (90,126) .. controls (90,128.76) and (87.76,131) .. (85,131) .. controls (82.24,131) and (80,128.76) .. (80,126) -- cycle ;
					\draw [line width=1.5]    (85,126) -- (205,126) ;
					\draw [line width=1.5]    (255,126) -- (325,126) ;
					\draw [line width=1.5]  [dash pattern={on 5.63pt off 4.5pt}]  (205,126) .. controls (140,59.5) and (214,58.5) .. (325,126) ;
					\draw [line width=1.5]  [dash pattern={on 5.63pt off 4.5pt}]  (205,126) .. controls (291,61.5) and (393,60.5) .. (325,126) ;
					\draw [line width=1.5]    (205,126) .. controls (242,158.5) and (289,160.5) .. (325,126) ;
				
				\draw (140,105) node [anchor=north west][inner sep=0.75pt]   [align=left] {$\bar 2$};
				\draw (233,62) node [anchor=north west][inner sep=0.75pt]   [align=left] {$\bar 4$};
				\draw (278,62) node [anchor=north west][inner sep=0.75pt]   [align=left] {$\bar 5$};
				\draw (263,132) node [anchor=north west][inner sep=0.75pt]   [align=left] {$\bar 1$};
				\draw (268,110) node [anchor=north west][inner sep=0.75pt]   [align=left] {$\bar 3$};

			\end{tikzpicture}
			
		\end{center}		
		The corresponding sf Brauer graph algebra is $A=kQ_{\Gamma,L}/I_{\Gamma,L}$ where the quiver $Q_{\Gamma,L}$ is
		$$
		\begin{tikzcd}
			& 1 \arrow[ld, "\alpha_1"'] \arrow[rd, "\beta_1"] &                         \\
			2 \arrow[rd, "\alpha_2"'] & 5 \arrow[u, "\alpha_4"]                         & 3 \arrow[ld, "\beta_2"] \\
			& 4 \arrow[u, "\alpha_3"]                         &                        
		\end{tikzcd}$$
		and the ideal $I_{\Gamma,L}$ is generated by the relations
		\begin{enumerate}[(i)]
			\item $\alpha_2\alpha_1=\beta_2\beta_1$;
			\item $\alpha_4\alpha_3\alpha_2\alpha_1\alpha_4=\alpha_3\alpha_2\alpha_1\alpha_4\alpha_3=\alpha_2\alpha_1\alpha_4\alpha_3\alpha_2=\alpha_1\alpha_4\alpha_3\alpha_2\alpha_1=0$,\\ $\alpha_4\alpha_3\beta_2\beta_1\alpha_4=\alpha_3\beta_2\beta_1\alpha_4\alpha_3=\beta_2\beta_1\alpha_4\alpha_3\beta_2=\beta_1\alpha_4\alpha_3\beta_2\beta_1=0$;
			\item $\alpha_{1}\alpha_{4}\alpha_{3}\beta_2=\beta_1\alpha_{4}\alpha_{3}\alpha_{2}=0$.
		\end{enumerate}
		The indecomposable projective module of $A$ are 
		$$P_1=\begin{array}{*{3}{lll}}
			&1&\\2&&3\\&4&\\&5&\\&1&
		\end{array},\quad P_2=\begin{array}{*{3}{lll}}
			&2&\\&4&\\&5&\\&1&\\&2&
		\end{array},\quad P_3=\begin{array}{*{3}{lll}}
			&3&\\&4&\\&5&\\&1&\\&3&
		\end{array},\quad P_4=\begin{array}{*{3}{lll}}
			&4&\\&5&\\&1&\\2&&3\\&4&
		\end{array}, \quad P_5=\begin{array}{*{3}{lll}}
			&5&\\&1&\\2&&3\\&4&\\&5&
		\end{array}.$$
		
		Indeed, by \cite[Example 6.7]{LL}, this sf Brauer graph algebra is derived equivalent to the algebra $B=\qa$ which is given by the quiver
		
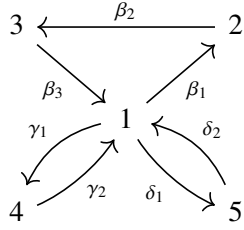
\[\begin{tikzcd}
			3 && 2 \\
			& 1 \\
			4 && 5
			\arrow["{\beta_3}"', from=1-1, to=2-2]
			\arrow["{\beta_2}"', from=1-3, to=1-1]
			\arrow["{\beta_1}"', from=2-2, to=1-3]
			\arrow["{\gamma_1}"', shift right, curve={height=6pt}, from=2-2, to=3-1]
			\arrow["{\delta_1}"', curve={height=6pt}, from=2-2, to=3-3]
			\arrow["{\gamma_2}"', curve={height=6pt}, from=3-1, to=2-2]
			\arrow["{\delta_2}"', shift right, curve={height=6pt}, from=3-3, to=2-2]
		\end{tikzcd}\]
		with relations $0=\beta_1\gamma_2=\gamma_1\beta_3=\beta_1\delta_2=\delta_1\beta_3=\gamma_1\delta_2=\delta_1\gamma_2=\beta_2\beta_1\beta_3\beta_2, \; \beta_3\beta_2\beta_1=\gamma_2\gamma_1=\delta_2\delta_1$ which is the self-injective algebra of finite representation type of type $(D_5,1,1)$ in \cite{BG}.
		
		Indeed, $B$ is a Brauer configuration algebra which is defined in \cite{GS2}. Thus $A$ is of finite representation type which is not a Brauer tree algebra. 
	\end{example}
	
	Now we begin to give some explicit description about sf Brauer graph algebras which are of finite representation type.
	
	\begin{definition}
		Let $\Gamma_L=(\Gamma=(V,H,s,\iota,\rho),L)$ be a labeled ribbon graph. Then we can define a new ribbon graph $\Gamma\backslash L=(V',H',s',\iota',\rho')$ as follows.
		\begin{itemize}
			\item $V'=V$, and $H'\subseteq H$ contains the half-edges which are not in any edges in $L$. 
			\item $s'$ and $\iota'$ are the restriction on $H'$.
			\item For each $h\in H'$, $\rho'(h)=\rho^i(h)$, where $i$ is the smallest positive integer, such that $\rho^i(h)$ is not a half-edge in any edges in $L$.
		\end{itemize}
	\end{definition}
	
	We show that the new defined ribbon graph can naturally induce a subalgebra of each sf Brauer graph algebra.
	
	\begin{proposition}\label{prop:eAe}
		Let $A$ be an sf Brauer graph algebra associated with the sf Brauer graph $(\Gamma_L,\m)$. Denote by $B$, the Brauer graph algebra corresponding to the Brauer graph $(\Gamma\backslash L,\m)$. Then $$B\cong eAe$$ as $k$-algebras, where $e=\sum_{\bar h \in E(\Gamma\backslash L)} e_{\bar h}$ and $e_{\bar h}$ is the idempotent corresponding to the edge $\bar h$ in $A$.
	\end{proposition}
	
	\begin{proof}
		It is straightforward since each vertex corresponding to a labeled edge in $L$ vanishes in $Q_{\Gamma\backslash L}$, and each single arrow starting from a labeled edge in $Q_{\Gamma,L}$ will been composed into new arrows in $Q_{\Gamma\backslash L}$. To be more specific, see Figure \ref{fig:iso-eAe}.
		\begin{figure}[ht]
			\centering
				\tikzset{every picture/.style={line width=0.75pt}} 
			
			\begin{tikzpicture}[x=0.75pt,y=0.75pt,yscale=-1,xscale=1]
				
				\draw  [fill={rgb, 255:red, 0; green, 0; blue, 0 }  ,fill opacity=1 ] (201,131) .. controls (201,128.24) and (203.24,126) .. (206,126) .. controls (208.76,126) and (211,128.24) .. (211,131) .. controls (211,133.76) and (208.76,136) .. (206,136) .. controls (203.24,136) and (201,133.76) .. (201,131) -- cycle ;
				\draw [line width=1.5]    (126,131) -- (206,131) ;
				\draw [line width=1.5]    (206,131) -- (286,131) ;
				\draw [line width=1.5]  [dash pattern={on 5.63pt off 4.5pt}]  (146,81.5) -- (206,131) ;
				\draw [line width=1.5]  [dash pattern={on 5.63pt off 4.5pt}]  (206,131) -- (206,58) ;
				\draw [line width=1.5]  [dash pattern={on 5.63pt off 4.5pt}]  (206,131) -- (259,79.5) ;
				\draw  [draw opacity=0] (166.38,125.49) .. controls (167.25,119.15) and (169.61,113.28) .. (173.1,108.24) -- (206,131) -- cycle ; \draw   (166.38,125.49) .. controls (167.25,119.15) and (169.61,113.28) .. (173.1,108.24) ;  
				\draw  [draw opacity=0] (178.69,101.77) .. controls (184.47,96.37) and (191.85,92.66) .. (200.05,91.44) -- (206,131) -- cycle ; \draw   (178.69,101.77) .. controls (184.47,96.37) and (191.85,92.66) .. (200.05,91.44) ;  
				\draw  [draw opacity=0] (238.6,107.82) .. controls (242.42,113.18) and (244.96,119.53) .. (245.74,126.42) -- (206,131) -- cycle ; \draw   (238.6,107.82) .. controls (242.42,113.18) and (244.96,119.53) .. (245.74,126.42) ;  
				\draw  [draw opacity=0][dash pattern={on 0.84pt off 2.51pt}] (210.8,91.28) .. controls (218.83,92.24) and (226.13,95.58) .. (231.97,100.57) -- (206,131) -- cycle ; \draw  [dash pattern={on 0.84pt off 2.51pt}] (210.8,91.28) .. controls (218.83,92.24) and (226.13,95.58) .. (231.97,100.57) ;  
				\draw  [draw opacity=0][dash pattern={on 0.84pt off 2.51pt}] (245.39,138.01) .. controls (242.07,156.76) and (225.7,171) .. (206,171) .. controls (185.98,171) and (169.4,156.3) .. (166.46,137.1) -- (206,131) -- cycle ; \draw  [dash pattern={on 0.84pt off 2.51pt}] (245.39,138.01) .. controls (242.07,156.76) and (225.7,171) .. (206,171) .. controls (185.98,171) and (169.4,156.3) .. (166.46,137.1) ;  
				\draw  [fill={rgb, 255:red, 0; green, 0; blue, 0 }  ,fill opacity=1 ] (431,131) .. controls (431,128.24) and (433.24,126) .. (436,126) .. controls (438.76,126) and (441,128.24) .. (441,131) .. controls (441,133.76) and (438.76,136) .. (436,136) .. controls (433.24,136) and (431,133.76) .. (431,131) -- cycle ;
				\draw [line width=1.5]    (356,131) -- (436,131) ;
				\draw [line width=1.5]    (436,131) -- (516,131) ;
				\draw  [draw opacity=0] (396.38,125.49) .. controls (399.06,106.01) and (415.78,91) .. (436,91) .. controls (455.7,91) and (472.07,105.24) .. (475.39,123.99) -- (436,131) -- cycle ; \draw   (396.38,125.49) .. controls (399.06,106.01) and (415.78,91) .. (436,91) .. controls (455.7,91) and (472.07,105.24) .. (475.39,123.99) ;  
				\draw  [draw opacity=0][dash pattern={on 0.84pt off 2.51pt}] (475.39,138.01) .. controls (472.07,156.76) and (455.7,171) .. (436,171) .. controls (415.98,171) and (399.4,156.3) .. (396.46,137.1) -- (436,131) -- cycle ; \draw  [dash pattern={on 0.84pt off 2.51pt}] (475.39,138.01) .. controls (472.07,156.76) and (455.7,171) .. (436,171) .. controls (415.98,171) and (399.4,156.3) .. (396.46,137.1) ;  
				\draw   (163.48,112.97) -- (173.4,107.67) -- (175.71,118.68) ;
				\draw   (191.96,85.82) -- (201.94,91.02) -- (194.16,99.14) ;
				\draw   (250.04,117.82) -- (245.53,128.13) -- (236.9,120.92) ;
				\draw   (480.04,114.82) -- (475.53,125.13) -- (466.9,117.92) ;
				
				\draw (111,122) node [anchor=north west][inner sep=0.75pt]   [align=left] {$h$};
				\draw (341,122) node [anchor=north west][inner sep=0.75pt]   [align=left] {$h$};
				\draw (143,104) node [anchor=north west][inner sep=0.75pt]   [align=left] {$\alpha_h$};
				\draw (174,72) node [anchor=north west][inner sep=0.75pt]   [align=left] {$\alpha_1$};
				\draw (255,104) node [anchor=north west][inner sep=0.75pt]   [align=left] {$\alpha_n$};
				\draw (430,70) node [anchor=north west][inner sep=0.75pt]   [align=left] {$\alpha_h$};

			\end{tikzpicture}
			\caption{$\alpha_h$ in $Q_{\Gamma,L}$ and in $Q_{\Gamma\backslash L}$.}
			\label{fig:iso-eAe}
		\end{figure}
		This isomorphism map each $\alpha_h$ in $Q_{\Gamma\backslash L}$ into the path $\alpha_n\cdots\alpha_1\alpha_h$ in $Q_{\Gamma,L}$, where $\alpha_1,\cdots,\alpha_n$ are arrows starting from some labeled edges in $\Gamma_L$.
	\end{proof}
	
	We have the following corollary.
	
	\begin{corollary}\label{LBTA is RFS}
		Let $A$ be an sf Brauer graph algebra associated with the sf Brauer graph $(\Gamma_L,\m)$. If $A$ is of finite representation type, then $(\Gamma\backslash L,\m)$ is a Brauer tree.
	\end{corollary}
	\begin{proof}
	By \cite{Bo}, if there exists an idempotent $e$ such that $eAe$ is of infinite representation type, then so is $A$.  Since $A$ is of representation type, by Proposition \ref{prop:eAe}, the Brauer graph algebra $B$ associated with $(\Gamma\backslash L,\m)$ is also of finite representation type. This is equal to say that the Brauer graph $(\Gamma\backslash L,\m)$ is a Brauer tree (see Theorem \ref{thm:BTA-RFS}).
	\end{proof}

	However, the converse of the above corollary does not hold. We provide the following example.

	\begin{example}\label{exa:non-rep-finite-FBTA}
		Let $\Gamma_L$ be an sf Brauer graph with multiplicity $1$ at each vertex which is represented as follows (with clockwise cyclic orientations at each vertex).
		
			\begin{center}

			\tikzset{every picture/.style={line width=0.75pt}} 
			
			\begin{tikzpicture}[x=0.75pt,y=0.75pt,yscale=-1,xscale=1]
					
					\draw  [fill={rgb, 255:red, 0; green, 0; blue, 0 }  ,fill opacity=1 ] (200,126) .. controls (200,123.24) and (202.24,121) .. (205,121) .. controls (207.76,121) and (210,123.24) .. (210,126) .. controls (210,128.76) and (207.76,131) .. (205,131) .. controls (202.24,131) and (200,128.76) .. (200,126) -- cycle ;
					\draw  [fill={rgb, 255:red, 0; green, 0; blue, 0 }  ,fill opacity=1 ] (320,126) .. controls (320,123.24) and (322.24,121) .. (325,121) .. controls (327.76,121) and (330,123.24) .. (330,126) .. controls (330,128.76) and (327.76,131) .. (325,131) .. controls (322.24,131) and (320,128.76) .. (320,126) -- cycle ;
					\draw  [fill={rgb, 255:red, 0; green, 0; blue, 0 }  ,fill opacity=1 ] (250,126) .. controls (250,123.24) and (252.24,121) .. (255,121) .. controls (257.76,121) and (260,123.24) .. (260,126) .. controls (260,128.76) and (257.76,131) .. (255,131) .. controls (252.24,131) and (250,128.76) .. (250,126) -- cycle ;
					\draw  [fill={rgb, 255:red, 0; green, 0; blue, 0 }  ,fill opacity=1 ] (80,126) .. controls (80,123.24) and (82.24,121) .. (85,121) .. controls (87.76,121) and (90,123.24) .. (90,126) .. controls (90,128.76) and (87.76,131) .. (85,131) .. controls (82.24,131) and (80,128.76) .. (80,126) -- cycle ;
					\draw [line width=1.5]    (85,126) -- (205,126) ;
					\draw [line width=1.5]    (130,206) -- (205,126) ;
					\draw  [fill={rgb, 255:red, 0; green, 0; blue, 0 },fill opacity=1]
(125,206) .. controls (125,203.24) and (127.24,201) .. (130,201)
.. controls (132.76,201) and (135,203.24) .. (135,206)
.. controls (135,208.76) and (132.76,211) .. (130,211)
.. controls (127.24,211) and (125,208.76) .. (125,206) -- cycle ;
					\draw [line width=1.5]    (255,126) -- (325,126) ;
					\draw [line width=1.5]  [dash pattern={on 5.63pt off 4.5pt}]  (205,126) .. controls (140,59.5) and (214,58.5) .. (325,126) ;
					\draw [line width=1.5]  [dash pattern={on 5.63pt off 4.5pt}]  (205,126) .. controls (291,61.5) and (393,60.5) .. (325,126) ;
					\draw [line width=1.5]    (205,126) .. controls (242,158.5) and (289,160.5) .. (325,126) ;
				
				\draw (140,105) node [anchor=north west][inner sep=0.75pt]   [align=left] {$\bar 2$};
				\draw (150,155) node [anchor=north west][inner sep=0.75pt]   [align=left] {$\bar 6$};
				\draw (233,62) node [anchor=north west][inner sep=0.75pt]   [align=left] {$\bar 4$};
				\draw (278,62) node [anchor=north west][inner sep=0.75pt]   [align=left] {$\bar 5$};
				\draw (263,132) node [anchor=north west][inner sep=0.75pt]   [align=left] {$\bar 1$};
				\draw (268,110) node [anchor=north west][inner sep=0.75pt]   [align=left] {$\bar 3$};

			\end{tikzpicture}
			
		\end{center}
		Then $(\Gamma\backslash L,\mathfrak m)$ is a Brauer tree. However, the corresponding sf Brauer graph algebra is $A=kQ_{\Gamma,L}/I_{\Gamma,L}$, where the quiver $Q_{\Gamma,L}$ is given by
\[
\begin{tikzcd}
6 \arrow[dd] & 1 \arrow[l] \arrow[rd] &              \\
             & 5 \arrow[u]            & 3 \arrow[ld] \\
2 \arrow[r]  & 4 \arrow[u]            &             
\end{tikzcd}
\]
Its universal cover contains a full subquiver of type $\widetilde E_6$. Hence, by \cite[Lemma 3.3]{Ga1981}, the algebra $A$ is of infinite representation type.
	\end{example}

	\subsection{Kauer moves}
	\
	\smallskip
	
	In this section, we want to show some derived equivalent examples of sf Brauer graph algebras. First of all, recall the two-term tilting complexes constructed by Okuyama \cite{Oku}. For a module $M$, denote by $P(M)$ the projective cover of $M$.
	
	\begin{theorem}\label{Okuyama}
		Let $A$ be a symmetric $k$-algebra. Let $I,I'$ be disjoint sets of simple $A$-modules such that $I\cup I'$ is a complete set of representatives of the isomorphism classes of simple $A$-modules.
		
		For any $S\in I$, let $T_S$ be the complex $Q\rightarrow P(S)$ with non-zero terms in degrees zero and one, and such that the nonzero differential is the right $\mathrm{add}(P(I'))$-approximation of $P(S)$ with $P(I')=\{P(S)\;|\;S\in I'\}$. For any $S'\in I'$ consider $P(S')$ as a complex concentrated in degree zero.
		
		Then the complex $T=(\bigoplus_{S\in I}T_S)\oplus(\bigoplus_{S'\in I'}P(S'))$ is a tilting complex for $A$.
	\end{theorem}
	
	The Okuyama tilting complexes for Brauer graph algebras are first constructed in \cite{Kauer}. Let $A$ be a Brauer graph algebra associated with a ribbon graph $\Gamma$. For $I=\{S_0\}$ where $S_0$ is a simple $A$-module, set $T=T_{S_0}\oplus(\bigoplus_{S\neq S_0}P(S))$ with $S$ simple. Then Kauer shows that the Brauer graph algebra $B=\mathrm{End}_{\mathscr{D}^b(A)}(T)$ has Brauer graph $\Gamma'=(\Gamma\backslash s)\cup s'$ where $s$ is the edge in $\Gamma$ corresponding to $S_0$ and where $s'$ is obtained by one of the local moves in Figure \ref{fig:Kauer-moves} on the corresponding ribbon graph $\Gamma$ (the orientations around all vertices are clockwise).
	
	
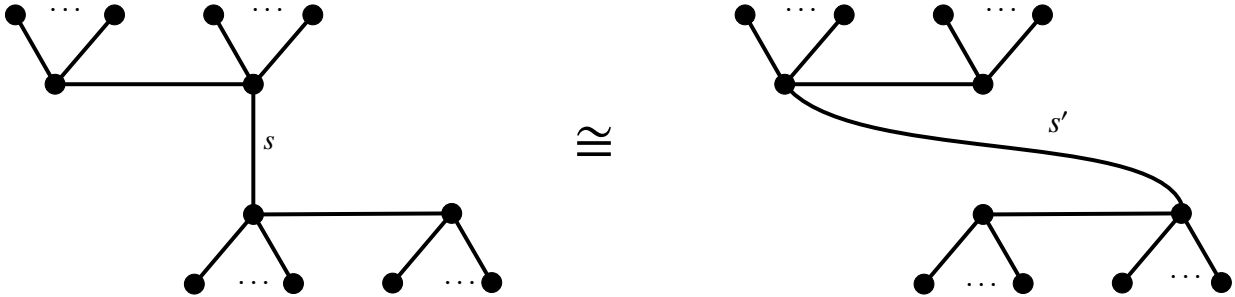
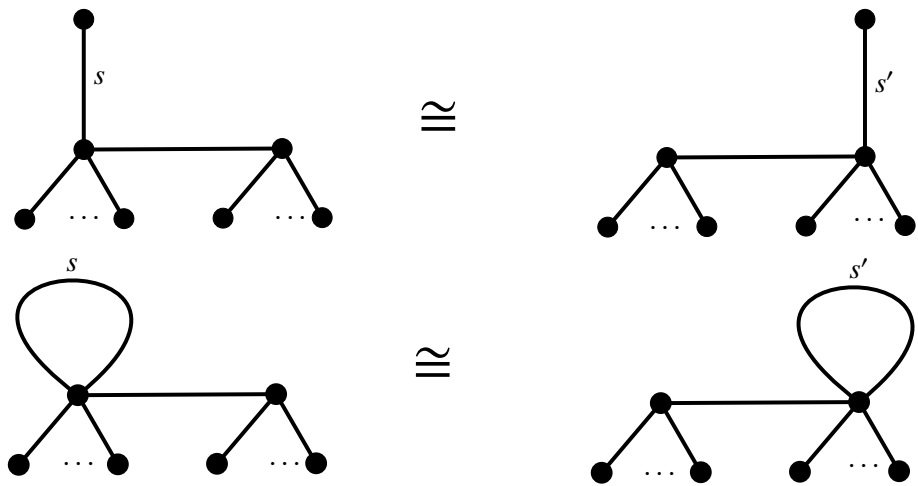
\begin{figure}[ht]
		\centering
		\begin{center}       
			\begin{tikzpicture}[x=0.75pt,y=0.75pt,yscale=-1,xscale=1]
				
				\draw  [fill={rgb, 255:red, 0; green, 0; blue, 0 }  ,fill opacity=1 ] (50,55) .. controls (50,52.24) and (52.24,50) .. (55,50) .. controls (57.76,50) and (60,52.24) .. (60,55) .. controls (60,57.76) and (57.76,60) .. (55,60) .. controls (52.24,60) and (50,57.76) .. (50,55) -- cycle ;
				\draw  [fill={rgb, 255:red, 0; green, 0; blue, 0 }  ,fill opacity=1 ] (100,55) .. controls (100,52.24) and (102.24,50) .. (105,50) .. controls (107.76,50) and (110,52.24) .. (110,55) .. controls (110,57.76) and (107.76,60) .. (105,60) .. controls (102.24,60) and (100,57.76) .. (100,55) -- cycle ;
				\draw [line width=1.5]    (55,55) -- (75,90) -- (105,55) ;
				\draw  [fill={rgb, 255:red, 0; green, 0; blue, 0 }  ,fill opacity=1 ] (70,90) .. controls (70,87.24) and (72.24,85) .. (75,85) .. controls (77.76,85) and (80,87.24) .. (80,90) .. controls (80,92.76) and (77.76,95) .. (75,95) .. controls (72.24,95) and (70,92.76) .. (70,90) -- cycle ;
				\draw  [fill={rgb, 255:red, 0; green, 0; blue, 0 }  ,fill opacity=1 ] (150,55) .. controls (150,52.24) and (152.24,50) .. (155,50) .. controls (157.76,50) and (160,52.24) .. (160,55) .. controls (160,57.76) and (157.76,60) .. (155,60) .. controls (152.24,60) and (150,57.76) .. (150,55) -- cycle ;
				\draw  [fill={rgb, 255:red, 0; green, 0; blue, 0 }  ,fill opacity=1 ] (200,55) .. controls (200,52.24) and (202.24,50) .. (205,50) .. controls (207.76,50) and (210,52.24) .. (210,55) .. controls (210,57.76) and (207.76,60) .. (205,60) .. controls (202.24,60) and (200,57.76) .. (200,55) -- cycle ;
				\draw [line width=1.5]    (155,55) -- (175,90) -- (205,55) ;
				\draw  [fill={rgb, 255:red, 0; green, 0; blue, 0 }  ,fill opacity=1 ] (170,90) .. controls (170,87.24) and (172.24,85) .. (175,85) .. controls (177.76,85) and (180,87.24) .. (180,90) .. controls (180,92.76) and (177.76,95) .. (175,95) .. controls (172.24,95) and (170,92.76) .. (170,90) -- cycle ;
				\draw [line width=1.5]    (75,90) -- (175,90) ;
				\draw [line width=1.5]    (175,90) -- (175,155.5) ;
				\draw  [fill={rgb, 255:red, 0; green, 0; blue, 0 }  ,fill opacity=1 ] (300.25,190.08) .. controls (300.26,192.84) and (298.04,195.09) .. (295.28,195.11) .. controls (292.51,195.12) and (290.26,192.89) .. (290.25,190.13) .. controls (290.23,187.37) and (292.46,185.12) .. (295.22,185.11) .. controls (297.98,185.09) and (300.23,187.32) .. (300.25,190.08) -- cycle ;
				\draw  [fill={rgb, 255:red, 0; green, 0; blue, 0 }  ,fill opacity=1 ] (250.25,190.34) .. controls (250.26,193.1) and (248.04,195.35) .. (245.28,195.37) .. controls (242.51,195.38) and (240.26,193.16) .. (240.25,190.39) .. controls (240.24,187.63) and (242.46,185.38) .. (245.22,185.37) .. controls (247.98,185.35) and (250.24,187.58) .. (250.25,190.34) -- cycle ;
				\draw [line width=1.5]    (295.25,190.11) -- (275.07,155.21) -- (245.25,190.37) ;
				\draw  [fill={rgb, 255:red, 0; green, 0; blue, 0 }  ,fill opacity=1 ] (280.07,155.19) .. controls (280.08,157.95) and (277.85,160.2) .. (275.09,160.21) .. controls (272.33,160.23) and (270.08,158) .. (270.07,155.24) .. controls (270.05,152.48) and (272.28,150.23) .. (275.04,150.21) .. controls (277.8,150.2) and (280.05,152.42) .. (280.07,155.19) -- cycle ;
				\draw  [fill={rgb, 255:red, 0; green, 0; blue, 0 }  ,fill opacity=1 ] (200.25,190.6) .. controls (200.26,193.37) and (198.04,195.62) .. (195.28,195.63) .. controls (192.52,195.64) and (190.26,193.42) .. (190.25,190.66) .. controls (190.24,187.9) and (192.46,185.64) .. (195.22,185.63) .. controls (197.99,185.62) and (200.24,187.84) .. (200.25,190.6) -- cycle ;
				\draw  [fill={rgb, 255:red, 0; green, 0; blue, 0 }  ,fill opacity=1 ] (150.25,190.87) .. controls (150.27,193.63) and (148.04,195.88) .. (145.28,195.89) .. controls (142.52,195.91) and (140.27,193.68) .. (140.25,190.92) .. controls (140.24,188.16) and (142.46,185.91) .. (145.22,185.89) .. controls (147.99,185.88) and (150.24,188.1) .. (150.25,190.87) -- cycle ;
				\draw [line width=1.5]    (195.25,190.63) -- (175.07,155.74) -- (145.25,190.89) ;
				\draw  [fill={rgb, 255:red, 0; green, 0; blue, 0 }  ,fill opacity=1 ] (180.07,155.71) .. controls (180.08,158.47) and (177.85,160.72) .. (175.09,160.74) .. controls (172.33,160.75) and (170.08,158.52) .. (170.07,155.76) .. controls (170.05,153) and (172.28,150.75) .. (175.04,150.74) .. controls (177.8,150.72) and (180.05,152.95) .. (180.07,155.71) -- cycle ;
				\draw [line width=1.5]    (275.07,155.21) -- (175.07,155.74) ;
				\draw  [fill={rgb, 255:red, 0; green, 0; blue, 0 }  ,fill opacity=1 ] (418,55) .. controls (418,52.24) and (420.24,50) .. (423,50) .. controls (425.76,50) and (428,52.24) .. (428,55) .. controls (428,57.76) and (425.76,60) .. (423,60) .. controls (420.24,60) and (418,57.76) .. (418,55) -- cycle ;
				\draw  [fill={rgb, 255:red, 0; green, 0; blue, 0 }  ,fill opacity=1 ] (468,55) .. controls (468,52.24) and (470.24,50) .. (473,50) .. controls (475.76,50) and (478,52.24) .. (478,55) .. controls (478,57.76) and (475.76,60) .. (473,60) .. controls (470.24,60) and (468,57.76) .. (468,55) -- cycle ;
				\draw [line width=1.5]    (423,55) -- (443,90) -- (473,55) ;
				\draw  [fill={rgb, 255:red, 0; green, 0; blue, 0 }  ,fill opacity=1 ] (438,90) .. controls (438,87.24) and (440.24,85) .. (443,85) .. controls (445.76,85) and (448,87.24) .. (448,90) .. controls (448,92.76) and (445.76,95) .. (443,95) .. controls (440.24,95) and (438,92.76) .. (438,90) -- cycle ;
				\draw  [fill={rgb, 255:red, 0; green, 0; blue, 0 }  ,fill opacity=1 ] (518,55) .. controls (518,52.24) and (520.24,50) .. (523,50) .. controls (525.76,50) and (528,52.24) .. (528,55) .. controls (528,57.76) and (525.76,60) .. (523,60) .. controls (520.24,60) and (518,57.76) .. (518,55) -- cycle ;
				\draw  [fill={rgb, 255:red, 0; green, 0; blue, 0 }  ,fill opacity=1 ] (568,55) .. controls (568,52.24) and (570.24,50) .. (573,50) .. controls (575.76,50) and (578,52.24) .. (578,55) .. controls (578,57.76) and (575.76,60) .. (573,60) .. controls (570.24,60) and (568,57.76) .. (568,55) -- cycle ;
				\draw [line width=1.5]    (523,55) -- (543,90) -- (573,55) ;
				\draw  [fill={rgb, 255:red, 0; green, 0; blue, 0 }  ,fill opacity=1 ] (538,90) .. controls (538,87.24) and (540.24,85) .. (543,85) .. controls (545.76,85) and (548,87.24) .. (548,90) .. controls (548,92.76) and (545.76,95) .. (543,95) .. controls (540.24,95) and (538,92.76) .. (538,90) -- cycle ;
				\draw [line width=1.5]    (443,90) -- (543,90) ;
				\draw  [fill={rgb, 255:red, 0; green, 0; blue, 0 }  ,fill opacity=1 ] (668.25,190.08) .. controls (668.26,192.84) and (666.04,195.09) .. (663.28,195.11) .. controls (660.51,195.12) and (658.26,192.89) .. (658.25,190.13) .. controls (658.23,187.37) and (660.46,185.12) .. (663.22,185.11) .. controls (665.98,185.09) and (668.23,187.32) .. (668.25,190.08) -- cycle ;
				\draw  [fill={rgb, 255:red, 0; green, 0; blue, 0 }  ,fill opacity=1 ] (618.25,190.34) .. controls (618.26,193.1) and (616.04,195.35) .. (613.28,195.37) .. controls (610.51,195.38) and (608.26,193.16) .. (608.25,190.39) .. controls (608.24,187.63) and (610.46,185.38) .. (613.22,185.37) .. controls (615.98,185.35) and (618.24,187.58) .. (618.25,190.34) -- cycle ;
				\draw [line width=1.5]    (663.25,190.11) -- (643.07,155.21) -- (613.25,190.37) ;
				\draw  [fill={rgb, 255:red, 0; green, 0; blue, 0 }  ,fill opacity=1 ] (648.07,155.19) .. controls (648.08,157.95) and (645.85,160.2) .. (643.09,160.21) .. controls (640.33,160.23) and (638.08,158) .. (638.07,155.24) .. controls (638.05,152.48) and (640.28,150.23) .. (643.04,150.21) .. controls (645.8,150.2) and (648.05,152.42) .. (648.07,155.19) -- cycle ;
				\draw  [fill={rgb, 255:red, 0; green, 0; blue, 0 }  ,fill opacity=1 ] (568.25,190.6) .. controls (568.26,193.37) and (566.04,195.62) .. (563.28,195.63) .. controls (560.52,195.64) and (558.26,193.42) .. (558.25,190.66) .. controls (558.24,187.9) and (560.46,185.64) .. (563.22,185.63) .. controls (565.99,185.62) and (568.24,187.84) .. (568.25,190.6) -- cycle ;
				\draw  [fill={rgb, 255:red, 0; green, 0; blue, 0 }  ,fill opacity=1 ] (518.25,190.87) .. controls (518.27,193.63) and (516.04,195.88) .. (513.28,195.89) .. controls (510.52,195.91) and (508.27,193.68) .. (508.25,190.92) .. controls (508.24,188.16) and (510.46,185.91) .. (513.22,185.89) .. controls (515.99,185.88) and (518.24,188.1) .. (518.25,190.87) -- cycle ;
				\draw [line width=1.5]    (563.25,190.63) -- (543.07,155.74) -- (513.25,190.89) ;
				\draw  [fill={rgb, 255:red, 0; green, 0; blue, 0 }  ,fill opacity=1 ] (548.07,155.71) .. controls (548.08,158.47) and (545.85,160.72) .. (543.09,160.74) .. controls (540.33,160.75) and (538.08,158.52) .. (538.07,155.76) .. controls (538.05,153) and (540.28,150.75) .. (543.04,150.74) .. controls (545.8,150.72) and (548.05,152.95) .. (548.07,155.71) -- cycle ;
				\draw [line width=1.5]    (643.07,155.21) -- (543.07,155.74) ;
				\draw [line width=1.5]    (643.04,150.21) .. controls (626,112.5) and (474,130.5) .. (443,90) ;
				
				\draw (178,115) node [anchor=north west][inner sep=0.75pt]   [align=left] {$s$};
				\draw (335,110) node [anchor=north west][inner sep=0.75pt] [font=\Huge]  [align=left] {$\cong$};
				\draw (574,102) node [anchor=north west][inner sep=0.75pt]   [align=left] {\large$s'$};
				\draw (71,50) node [anchor=north west][inner sep=0.75pt]   [align=left] {$\cdots$};
				\draw (173,50) node [anchor=north west][inner sep=0.75pt]   [align=left] {$\cdots$};
				\draw (166,188) node [anchor=north west][inner sep=0.75pt]   [align=left] {$\cdots$};
				\draw (270,188) node [anchor=north west][inner sep=0.75pt]   [align=left] {$\cdots$};
				\draw (442,50) node [anchor=north west][inner sep=0.75pt]   [align=left] {$\cdots$};
				\draw (543,50) node [anchor=north west][inner sep=0.75pt]   [align=left] {$\cdots$};
				\draw (533,189) node [anchor=north west][inner sep=0.75pt]   [align=left] {$\cdots$};
				\draw (636,185) node [anchor=north west][inner sep=0.75pt]   [align=left] {$\cdots$};

			\end{tikzpicture}

			\smallskip

			\begin{tikzpicture}[x=0.75pt,y=0.75pt,yscale=-1,xscale=1]
				
				\draw  [fill={rgb, 255:red, 0; green, 0; blue, 0 }  ,fill opacity=1 ] (170,90) .. controls (170,87.24) and (172.24,85) .. (175,85) .. controls (177.76,85) and (180,87.24) .. (180,90) .. controls (180,92.76) and (177.76,95) .. (175,95) .. controls (172.24,95) and (170,92.76) .. (170,90) -- cycle ;
				\draw [line width=1.5]    (175,90) -- (175,155.5) ;
				\draw  [fill={rgb, 255:red, 0; green, 0; blue, 0 }  ,fill opacity=1 ] (300.25,190.08) .. controls (300.26,192.84) and (298.04,195.09) .. (295.28,195.11) .. controls (292.51,195.12) and (290.26,192.89) .. (290.25,190.13) .. controls (290.23,187.37) and (292.46,185.12) .. (295.22,185.11) .. controls (297.98,185.09) and (300.23,187.32) .. (300.25,190.08) -- cycle ;
				\draw  [fill={rgb, 255:red, 0; green, 0; blue, 0 }  ,fill opacity=1 ] (250.25,190.34) .. controls (250.26,193.1) and (248.04,195.35) .. (245.28,195.37) .. controls (242.51,195.38) and (240.26,193.16) .. (240.25,190.39) .. controls (240.24,187.63) and (242.46,185.38) .. (245.22,185.37) .. controls (247.98,185.35) and (250.24,187.58) .. (250.25,190.34) -- cycle ;
				\draw [line width=1.5]    (295.25,190.11) -- (275.07,155.21) -- (245.25,190.37) ;
				\draw  [fill={rgb, 255:red, 0; green, 0; blue, 0 }  ,fill opacity=1 ] (280.07,155.19) .. controls (280.08,157.95) and (277.85,160.2) .. (275.09,160.21) .. controls (272.33,160.23) and (270.08,158) .. (270.07,155.24) .. controls (270.05,152.48) and (272.28,150.23) .. (275.04,150.21) .. controls (277.8,150.2) and (280.05,152.42) .. (280.07,155.19) -- cycle ;
				\draw  [fill={rgb, 255:red, 0; green, 0; blue, 0 }  ,fill opacity=1 ] (200.25,190.6) .. controls (200.26,193.37) and (198.04,195.62) .. (195.28,195.63) .. controls (192.52,195.64) and (190.26,193.42) .. (190.25,190.66) .. controls (190.24,187.9) and (192.46,185.64) .. (195.22,185.63) .. controls (197.99,185.62) and (200.24,187.84) .. (200.25,190.6) -- cycle ;
				\draw  [fill={rgb, 255:red, 0; green, 0; blue, 0 }  ,fill opacity=1 ] (150.25,190.87) .. controls (150.27,193.63) and (148.04,195.88) .. (145.28,195.89) .. controls (142.52,195.91) and (140.27,193.68) .. (140.25,190.92) .. controls (140.24,188.16) and (142.46,185.91) .. (145.22,185.89) .. controls (147.99,185.88) and (150.24,188.1) .. (150.25,190.87) -- cycle ;
				\draw [line width=1.5]    (195.25,190.63) -- (175.07,155.74) -- (145.25,190.89) ;
				\draw  [fill={rgb, 255:red, 0; green, 0; blue, 0 }  ,fill opacity=1 ] (180.07,155.71) .. controls (180.08,158.47) and (177.85,160.72) .. (175.09,160.74) .. controls (172.33,160.75) and (170.08,158.52) .. (170.07,155.76) .. controls (170.05,153) and (172.28,150.75) .. (175.04,150.74) .. controls (177.8,150.72) and (180.05,152.95) .. (180.07,155.71) -- cycle ;
				\draw [line width=1.5]    (275.07,155.21) -- (175.07,155.74) ;
				\draw  [fill={rgb, 255:red, 0; green, 0; blue, 0 }  ,fill opacity=1 ] (594.25,194.08) .. controls (594.26,196.84) and (592.04,199.09) .. (589.28,199.11) .. controls (586.51,199.12) and (584.26,196.89) .. (584.25,194.13) .. controls (584.23,191.37) and (586.46,189.12) .. (589.22,189.11) .. controls (591.98,189.09) and (594.23,191.32) .. (594.25,194.08) -- cycle ;
				\draw  [fill={rgb, 255:red, 0; green, 0; blue, 0 }  ,fill opacity=1 ] (544.25,194.34) .. controls (544.26,197.1) and (542.04,199.35) .. (539.28,199.37) .. controls (536.51,199.38) and (534.26,197.16) .. (534.25,194.39) .. controls (534.24,191.63) and (536.46,189.38) .. (539.22,189.37) .. controls (541.98,189.35) and (544.24,191.58) .. (544.25,194.34) -- cycle ;
				\draw [line width=1.5]    (589.25,194.11) -- (569.07,159.21) -- (539.25,194.37) ;
				\draw  [fill={rgb, 255:red, 0; green, 0; blue, 0 }  ,fill opacity=1 ] (574.07,159.19) .. controls (574.08,161.95) and (571.85,164.2) .. (569.09,164.21) .. controls (566.33,164.23) and (564.08,162) .. (564.07,159.24) .. controls (564.05,156.48) and (566.28,154.23) .. (569.04,154.21) .. controls (571.8,154.2) and (574.05,156.42) .. (574.07,159.19) -- cycle ;
				\draw  [fill={rgb, 255:red, 0; green, 0; blue, 0 }  ,fill opacity=1 ] (494.25,194.6) .. controls (494.26,197.37) and (492.04,199.62) .. (489.28,199.63) .. controls (486.52,199.64) and (484.26,197.42) .. (484.25,194.66) .. controls (484.24,191.9) and (486.46,189.64) .. (489.22,189.63) .. controls (491.99,189.62) and (494.24,191.84) .. (494.25,194.6) -- cycle ;
				\draw  [fill={rgb, 255:red, 0; green, 0; blue, 0 }  ,fill opacity=1 ] (444.25,194.87) .. controls (444.27,197.63) and (442.04,199.88) .. (439.28,199.89) .. controls (436.52,199.91) and (434.27,197.68) .. (434.25,194.92) .. controls (434.24,192.16) and (436.46,189.91) .. (439.22,189.89) .. controls (441.99,189.88) and (444.24,192.1) .. (444.25,194.87) -- cycle ;
				\draw [line width=1.5]    (489.25,194.63) -- (469.07,159.74) -- (439.25,194.89) ;
				\draw  [fill={rgb, 255:red, 0; green, 0; blue, 0 }  ,fill opacity=1 ] (474.07,159.71) .. controls (474.08,162.47) and (471.85,164.72) .. (469.09,164.74) .. controls (466.33,164.75) and (464.08,162.52) .. (464.07,159.76) .. controls (464.05,157) and (466.28,154.75) .. (469.04,154.74) .. controls (471.8,154.72) and (474.05,156.95) .. (474.07,159.71) -- cycle ;
				\draw [line width=1.5]    (569.07,159.21) -- (469.07,159.74) ;
				\draw  [fill={rgb, 255:red, 0; green, 0; blue, 0 }  ,fill opacity=1 ] (564,90) .. controls (564,87.24) and (566.24,85) .. (569,85) .. controls (571.76,85) and (574,87.24) .. (574,90) .. controls (574,92.76) and (571.76,95) .. (569,95) .. controls (566.24,95) and (564,92.76) .. (564,90) -- cycle ;
				\draw [line width=1.5]    (569,90) -- (569,155.5) ;
				
				\draw (178,115) node [anchor=north west][inner sep=0.75pt]   [align=left] {$s$};
				\draw (342,130) node [anchor=north west][inner sep=0.75pt]  [font=\Huge] [align=left] {$\cong$};
				\draw (166,188) node [anchor=north west][inner sep=0.75pt]   [align=left] {$\cdots$};
				\draw (270,188) node [anchor=north west][inner sep=0.75pt]   [align=left] {$\cdots$};
				\draw (459,193) node [anchor=north west][inner sep=0.75pt]   [align=left] {$\cdots$};
				\draw (562,189) node [anchor=north west][inner sep=0.75pt]   [align=left] {$\cdots$};
				\draw (572,115) node [anchor=north west][inner sep=0.75pt]   [align=left] {$s'$};

			\end{tikzpicture}

			\tikzset{every picture/.style={line width=0.75pt}} 
			
			\begin{tikzpicture}[x=0.75pt,y=0.75pt,yscale=-1,xscale=1]
				
				\draw  [fill={rgb, 255:red, 0; green, 0; blue, 0 }  ,fill opacity=1 ] (300.25,190.08) .. controls (300.26,192.84) and (298.04,195.09) .. (295.28,195.11) .. controls (292.51,195.12) and (290.26,192.89) .. (290.25,190.13) .. controls (290.23,187.37) and (292.46,185.12) .. (295.22,185.11) .. controls (297.98,185.09) and (300.23,187.32) .. (300.25,190.08) -- cycle ;
				\draw  [fill={rgb, 255:red, 0; green, 0; blue, 0 }  ,fill opacity=1 ] (250.25,190.34) .. controls (250.26,193.1) and (248.04,195.35) .. (245.28,195.37) .. controls (242.51,195.38) and (240.26,193.16) .. (240.25,190.39) .. controls (240.24,187.63) and (242.46,185.38) .. (245.22,185.37) .. controls (247.98,185.35) and (250.24,187.58) .. (250.25,190.34) -- cycle ;
				\draw [line width=1.5]    (295.25,190.11) -- (275.07,155.21) -- (245.25,190.37) ;
				\draw  [fill={rgb, 255:red, 0; green, 0; blue, 0 }  ,fill opacity=1 ] (280.07,155.19) .. controls (280.08,157.95) and (277.85,160.2) .. (275.09,160.21) .. controls (272.33,160.23) and (270.08,158) .. (270.07,155.24) .. controls (270.05,152.48) and (272.28,150.23) .. (275.04,150.21) .. controls (277.8,150.2) and (280.05,152.42) .. (280.07,155.19) -- cycle ;
				\draw  [fill={rgb, 255:red, 0; green, 0; blue, 0 }  ,fill opacity=1 ] (200.25,190.6) .. controls (200.26,193.37) and (198.04,195.62) .. (195.28,195.63) .. controls (192.52,195.64) and (190.26,193.42) .. (190.25,190.66) .. controls (190.24,187.9) and (192.46,185.64) .. (195.22,185.63) .. controls (197.99,185.62) and (200.24,187.84) .. (200.25,190.6) -- cycle ;
				\draw  [fill={rgb, 255:red, 0; green, 0; blue, 0 }  ,fill opacity=1 ] (150.25,190.87) .. controls (150.27,193.63) and (148.04,195.88) .. (145.28,195.89) .. controls (142.52,195.91) and (140.27,193.68) .. (140.25,190.92) .. controls (140.24,188.16) and (142.46,185.91) .. (145.22,185.89) .. controls (147.99,185.88) and (150.24,188.1) .. (150.25,190.87) -- cycle ;
				\draw [line width=1.5]    (195.25,190.63) -- (175.07,155.74) -- (145.25,190.89) ;
				\draw  [fill={rgb, 255:red, 0; green, 0; blue, 0 }  ,fill opacity=1 ] (180.07,155.71) .. controls (180.08,158.47) and (177.85,160.72) .. (175.09,160.74) .. controls (172.33,160.75) and (170.08,158.52) .. (170.07,155.76) .. controls (170.05,153) and (172.28,150.75) .. (175.04,150.74) .. controls (177.8,150.72) and (180.05,152.95) .. (180.07,155.71) -- cycle ;
				\draw [line width=1.5]    (275.07,155.21) -- (175.07,155.74) ;
				\draw  [fill={rgb, 255:red, 0; green, 0; blue, 0 }  ,fill opacity=1 ] (594.25,194.08) .. controls (594.26,196.84) and (592.04,199.09) .. (589.28,199.11) .. controls (586.51,199.12) and (584.26,196.89) .. (584.25,194.13) .. controls (584.23,191.37) and (586.46,189.12) .. (589.22,189.11) .. controls (591.98,189.09) and (594.23,191.32) .. (594.25,194.08) -- cycle ;
				\draw  [fill={rgb, 255:red, 0; green, 0; blue, 0 }  ,fill opacity=1 ] (544.25,194.34) .. controls (544.26,197.1) and (542.04,199.35) .. (539.28,199.37) .. controls (536.51,199.38) and (534.26,197.16) .. (534.25,194.39) .. controls (534.24,191.63) and (536.46,189.38) .. (539.22,189.37) .. controls (541.98,189.35) and (544.24,191.58) .. (544.25,194.34) -- cycle ;
				\draw [line width=1.5]    (589.25,194.11) -- (569.07,159.21) -- (539.25,194.37) ;
				\draw  [fill={rgb, 255:red, 0; green, 0; blue, 0 }  ,fill opacity=1 ] (574.07,159.19) .. controls (574.08,161.95) and (571.85,164.2) .. (569.09,164.21) .. controls (566.33,164.23) and (564.08,162) .. (564.07,159.24) .. controls (564.05,156.48) and (566.28,154.23) .. (569.04,154.21) .. controls (571.8,154.2) and (574.05,156.42) .. (574.07,159.19) -- cycle ;
				\draw  [fill={rgb, 255:red, 0; green, 0; blue, 0 }  ,fill opacity=1 ] (494.25,194.6) .. controls (494.26,197.37) and (492.04,199.62) .. (489.28,199.63) .. controls (486.52,199.64) and (484.26,197.42) .. (484.25,194.66) .. controls (484.24,191.9) and (486.46,189.64) .. (489.22,189.63) .. controls (491.99,189.62) and (494.24,191.84) .. (494.25,194.6) -- cycle ;
				\draw  [fill={rgb, 255:red, 0; green, 0; blue, 0 }  ,fill opacity=1 ] (444.25,194.87) .. controls (444.27,197.63) and (442.04,199.88) .. (439.28,199.89) .. controls (436.52,199.91) and (434.27,197.68) .. (434.25,194.92) .. controls (434.24,192.16) and (436.46,189.91) .. (439.22,189.89) .. controls (441.99,189.88) and (444.24,192.1) .. (444.25,194.87) -- cycle ;
				\draw [line width=1.5]    (489.25,194.63) -- (469.07,159.74) -- (439.25,194.89) ;
				\draw  [fill={rgb, 255:red, 0; green, 0; blue, 0 }  ,fill opacity=1 ] (474.07,159.71) .. controls (474.08,162.47) and (471.85,164.72) .. (469.09,164.74) .. controls (466.33,164.75) and (464.08,162.52) .. (464.07,159.76) .. controls (464.05,157) and (466.28,154.75) .. (469.04,154.74) .. controls (471.8,154.72) and (474.05,156.95) .. (474.07,159.71) -- cycle ;
				\draw [line width=1.5]    (569.07,159.21) -- (469.07,159.74) ;
				\draw [line width=1.5]    (175.07,155.74) .. controls (271,78.5) and (72,78.5) .. (175.07,155.74) ;
				\draw [line width=1.5]    (569.07,159.21) .. controls (665,81.98) and (466,81.98) .. (569.07,159.21) ;
				
				\draw (342,130) node [anchor=north west][inner sep=0.75pt]  [font=\Huge] [align=left] {$\cong$};
				\draw (166,188) node [anchor=north west][inner sep=0.75pt]   [align=left] {$\cdots$};
				\draw (270,188) node [anchor=north west][inner sep=0.75pt]   [align=left] {$\cdots$};
				\draw (459,193) node [anchor=north west][inner sep=0.75pt]   [align=left] {$\cdots$};
				\draw (562,189) node [anchor=north west][inner sep=0.75pt]   [align=left] {$\cdots$};
				\draw (167,85) node [anchor=north west][inner sep=0.75pt]   [align=left] {$s$};
				\draw (562,85) node [anchor=north west][inner sep=0.75pt]   [align=left] {$s'$};

			\end{tikzpicture}    
			\smallskip	
		\end{center}
					\caption{Kauer moves on ribbon graphs.}
		\label{fig:Kauer-moves}	
	\end{figure}	
	We call such a local move in Figure \ref{fig:Kauer-moves} a {\it Kauer move} at $s$. Indeed, Kauer moves give a explicit description of the tilting mutations of Brauer graph algebras. Now we generalized these local moves into sf Brauer graph.
	\begin{definition}
		Let $s$ and $s'$ be an edge which is not labeled in two labeled ribbon graphs $\Gamma_L$ and $\Gamma_L'$, respectively. If the edges whom the moves have passed through in local moves in Figure \ref{fig:Kauer-moves} are not labeled, then these moves induced moves from $\Gamma_L$ to $\Gamma_L'$. We also call these moves by {\it Kauer moves} on the labeled ribbon graph $\Gamma_L$.
	\end{definition}
	
	We prove the following theorem. 
		
	\begin{theorem}\label{Kauer move}
		Let $A$ and $B$ be two sf Brauer graph algebras associated with sf Brauer graphs $(\Gamma^A_L,\m^A)$ and $(\Gamma^B_{L'},\m^B)$, respectively. Assume that the vertex sets of $\Gamma^A$ and $\Gamma^B$ coincide, and that $\m^A=\m^B$. If $\Gamma^A_L$ can be obtained from $\Gamma^B_{L'}$ via a finite sequence of Kauer moves, then $A$ and $B$ are derived equivalent.
	\end{theorem}
	
	\begin{proof}
		We just need to verify local moves in Figure \ref{fig:Kauer-moves} induces two derived equivalent algebras. By Theorem \ref{Okuyama}, we can construct the tilting complex from left to right with the partition of simple modules $I=\{S_0\}$ where $S_0$ is the simple module corresponding to $s$ in $\Gamma$. The verification is same in \cite[Lemma 3.5]{Kauer}. We give the explicit computation of the first type precisely.
		
		Denote the sf Brauer graph algebra which is induced by the left labeled ribbon graph (resp.\ the right one) in the first local move in Figure \ref{fig:Kauer-moves} by $A$ (resp.\ $B$). To streamline the notation, we set $n_i=\mathrm{val}(v_i)$. Give the following notations on the sf Brauer graph. To be more specific, $P_s,P_1,P_2,Q_2,\cdots, Q_{n_2}, R_2,\cdots,R_{n_1-1}$ are indecomposable projective $A$-module corresponding to the edges in the sf Brauer graph.
		
		
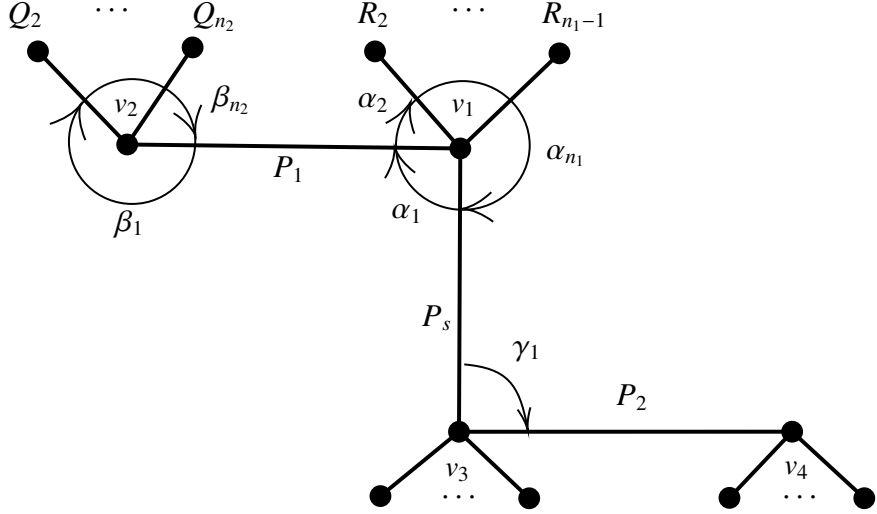
\begin{figure}
			\centering
			
			\tikzset{every picture/.style={line width=0.75pt}} 
			
			\begin{tikzpicture}[x=0.75pt,y=0.75pt,yscale=-1,xscale=1]
				
				\draw  [fill={rgb, 255:red, 0; green, 0; blue, 0 }  ,fill opacity=1 ] (223,94.5) .. controls (223,91.74) and (225.24,89.5) .. (228,89.5) .. controls (230.76,89.5) and (233,91.74) .. (233,94.5) .. controls (233,97.26) and (230.76,99.5) .. (228,99.5) .. controls (225.24,99.5) and (223,97.26) .. (223,94.5) -- cycle ;
				\draw [line width=1.5]    (228,94.5) -- (396,96.5) ;
				\draw [line width=1.5]    (183,47.5) -- (228,94.5) ;
				\draw [line width=1.5]    (228,94.5) -- (261,45.5) ;
				\draw [line width=1.5]    (353,47.5) -- (396,96.5) ;
				\draw [line width=1.5]    (446,48.5) -- (396,96.5) ;
				\draw [line width=1.5]    (396,96.5) -- (395.34,239.49) ;
				\draw [line width=1.5]    (563.36,239.46) -- (395.34,239.49) ;
				\draw [line width=1.5]    (599.76,273.03) -- (563.36,239.46) ;
				\draw [line width=1.5]    (563.36,239.46) -- (531.75,272.85) ;
				\draw [line width=1.5]    (430.75,273.06) -- (395.34,239.49) ;
				\draw [line width=1.5]    (355.73,271.96) -- (395.34,239.49) ;
				\draw   (198.62,93.29) .. controls (198.46,75.89) and (212.59,61.66) .. (230.18,61.5) .. controls (247.77,61.34) and (262.17,75.32) .. (262.33,92.71) .. controls (262.49,110.11) and (248.37,124.34) .. (230.77,124.5) .. controls (213.18,124.66) and (198.79,110.68) .. (198.62,93.29) -- cycle ;
				\draw   (189,84.12) .. controls (196.11,81.82) and (201.24,78.61) .. (204.37,74.48) .. controls (203.23,79.52) and (204.09,85.46) .. (206.93,92.3) ;
				\draw   (265.43,74.52) .. controls (262.87,80.1) and (261.87,85.21) .. (262.43,89.87) .. controls (260.31,85.67) and (256.63,81.95) .. (251.39,78.68) ;
				\draw   (363.56,95.28) .. controls (363.39,77.34) and (378.35,62.67) .. (396.97,62.5) .. controls (415.59,62.34) and (430.83,76.74) .. (431,94.68) .. controls (431.17,112.62) and (416.21,127.3) .. (397.59,127.46) .. controls (378.97,127.63) and (363.73,113.22) .. (363.56,95.28) -- cycle ;
				\draw   (358,111.52) .. controls (361.31,106.06) and (362.93,100.92) .. (362.85,96.08) .. controls (364.62,100.61) and (368.08,104.82) .. (373.24,108.73) ;
				\draw   (412.59,133.5) .. controls (407.38,129.65) and (402.29,127.46) .. (397.3,126.93) .. controls (402.18,125.81) and (406.95,123.03) .. (411.62,118.59) ;
				\draw   (358.39,83.72) .. controls (364.18,80.75) and (368.37,77.19) .. (370.93,73.04) .. controls (369.97,77.78) and (370.61,83.12) .. (372.87,89.06) ;
				\draw  [fill={rgb, 255:red, 0; green, 0; blue, 0 }  ,fill opacity=1 ] (391,96.5) .. controls (391,93.74) and (393.24,91.5) .. (396,91.5) .. controls (398.76,91.5) and (401,93.74) .. (401,96.5) .. controls (401,99.26) and (398.76,101.5) .. (396,101.5) .. controls (393.24,101.5) and (391,99.26) .. (391,96.5) -- cycle ;
				\draw  [fill={rgb, 255:red, 0; green, 0; blue, 0 }  ,fill opacity=1 ] (425.75,273.06) .. controls (425.75,270.3) and (427.99,268.06) .. (430.75,268.06) .. controls (433.51,268.06) and (435.75,270.3) .. (435.75,273.06) .. controls (435.75,275.82) and (433.51,278.06) .. (430.75,278.06) .. controls (427.99,278.06) and (425.75,275.82) .. (425.75,273.06) -- cycle ;
				\draw  [fill={rgb, 255:red, 0; green, 0; blue, 0 }  ,fill opacity=1 ] (558.36,239.46) .. controls (558.36,236.7) and (560.59,234.46) .. (563.36,234.46) .. controls (566.12,234.46) and (568.36,236.7) .. (568.36,239.46) .. controls (568.36,242.22) and (566.12,244.46) .. (563.36,244.46) .. controls (560.59,244.46) and (558.36,242.22) .. (558.36,239.46) -- cycle ;
				\draw  [fill={rgb, 255:red, 0; green, 0; blue, 0 }  ,fill opacity=1 ] (390.34,239.49) .. controls (390.34,236.72) and (392.58,234.49) .. (395.34,234.49) .. controls (398.1,234.49) and (400.34,236.72) .. (400.34,239.49) .. controls (400.34,242.25) and (398.1,244.49) .. (395.34,244.49) .. controls (392.58,244.49) and (390.34,242.25) .. (390.34,239.49) -- cycle ;
				\draw  [fill={rgb, 255:red, 0; green, 0; blue, 0 }  ,fill opacity=1 ] (350.73,271.96) .. controls (350.73,269.2) and (352.97,266.96) .. (355.73,266.96) .. controls (358.49,266.96) and (360.73,269.2) .. (360.73,271.96) .. controls (360.73,274.73) and (358.49,276.96) .. (355.73,276.96) .. controls (352.97,276.96) and (350.73,274.73) .. (350.73,271.96) -- cycle ;
				\draw  [fill={rgb, 255:red, 0; green, 0; blue, 0 }  ,fill opacity=1 ] (178,47.5) .. controls (178,44.74) and (180.24,42.5) .. (183,42.5) .. controls (185.76,42.5) and (188,44.74) .. (188,47.5) .. controls (188,50.26) and (185.76,52.5) .. (183,52.5) .. controls (180.24,52.5) and (178,50.26) .. (178,47.5) -- cycle ;
				\draw  [fill={rgb, 255:red, 0; green, 0; blue, 0 }  ,fill opacity=1 ] (256,45.5) .. controls (256,42.74) and (258.24,40.5) .. (261,40.5) .. controls (263.76,40.5) and (266,42.74) .. (266,45.5) .. controls (266,48.26) and (263.76,50.5) .. (261,50.5) .. controls (258.24,50.5) and (256,48.26) .. (256,45.5) -- cycle ;
				\draw  [fill={rgb, 255:red, 0; green, 0; blue, 0 }  ,fill opacity=1 ] (348,47.5) .. controls (348,44.74) and (350.24,42.5) .. (353,42.5) .. controls (355.76,42.5) and (358,44.74) .. (358,47.5) .. controls (358,50.26) and (355.76,52.5) .. (353,52.5) .. controls (350.24,52.5) and (348,50.26) .. (348,47.5) -- cycle ;
				\draw  [fill={rgb, 255:red, 0; green, 0; blue, 0 }  ,fill opacity=1 ] (441,48.5) .. controls (441,45.74) and (443.24,43.5) .. (446,43.5) .. controls (448.76,43.5) and (451,45.74) .. (451,48.5) .. controls (451,51.26) and (448.76,53.5) .. (446,53.5) .. controls (443.24,53.5) and (441,51.26) .. (441,48.5) -- cycle ;
				\draw  [fill={rgb, 255:red, 0; green, 0; blue, 0 }  ,fill opacity=1 ] (526.75,272.85) .. controls (526.75,270.08) and (528.99,267.85) .. (531.75,267.85) .. controls (534.52,267.85) and (536.75,270.08) .. (536.75,272.85) .. controls (536.75,275.61) and (534.52,277.85) .. (531.75,277.85) .. controls (528.99,277.85) and (526.75,275.61) .. (526.75,272.85) -- cycle ;
				\draw  [fill={rgb, 255:red, 0; green, 0; blue, 0 }  ,fill opacity=1 ] (594.76,273.03) .. controls (594.76,270.26) and (597,268.03) .. (599.76,268.03) .. controls (602.52,268.03) and (604.76,270.26) .. (604.76,273.03) .. controls (604.76,275.79) and (602.52,278.03) .. (599.76,278.03) .. controls (597,278.03) and (594.76,275.79) .. (594.76,273.03) -- cycle ;
				\draw    (398,205.5) .. controls (419.45,207.45) and (425.69,215.1) .. (429.7,236.8) ;
				\draw [shift={(430,238.5)}, rotate = 260.13] [color={rgb, 255:red, 0; green, 0; blue, 0 }  ][line width=0.75]    (10.93,-3.29) .. controls (6.95,-1.4) and (3.31,-0.3) .. (0,0) .. controls (3.31,0.3) and (6.95,1.4) .. (10.93,3.29)   ;
				
				\draw (210,24) node [anchor=north west][inner sep=0.75pt]  [font=\large] [align=left] {$\cdots$};
				\draw (375,174) node [anchor=north west][inner sep=0.75pt]  [font=\large] [align=left] {$P_s$};
				\draw (166,21) node [anchor=north west][inner sep=0.75pt]  [font=\large] [align=left] {$Q_2$};
				\draw (259,21) node [anchor=north west][inner sep=0.75pt]  [font=\large] [align=left] {$Q_{n_2}$};
				\draw (300,99) node [anchor=north west][inner sep=0.75pt]  [font=\large] [align=left] {$P_1$};
				\draw (473,214) node [anchor=north west][inner sep=0.75pt]  [font=\large] [align=left] {$P_2$};
				\draw (220,70) node [anchor=north west][inner sep=0.75pt]  [font=\normalsize] [align=left] {$v_2$};
				\draw (390,70) node [anchor=north west][inner sep=0.75pt]  [font=\normalsize] [align=left] {$v_1$};
				\draw (387,254) node [anchor=north west][inner sep=0.75pt]  [font=\normalsize] [align=left] {$v_3$};
				\draw (558,253) node [anchor=north west][inner sep=0.75pt]  [font=\normalsize] [align=left] {$v_4$};
				\draw (390,24) node [anchor=north west][inner sep=0.75pt]  [font=\large] [align=left] {$\cdots$};
				\draw (343,21) node [anchor=north west][inner sep=0.75pt]  [font=\large] [align=left] {$R_2$};
				\draw (436,21) node [anchor=north west][inner sep=0.75pt]  [font=\large] [align=left] {$R_{n_1-1}$};
				\draw (385,270) node [anchor=north west][inner sep=0.75pt]  [font=\large] [align=left] {$\cdots$};
				\draw (557,270) node [anchor=north west][inner sep=0.75pt]  [font=\large] [align=left] {$\cdots$};
				\draw (221,126) node [anchor=north west][inner sep=0.75pt]  [font=\large] [align=left] {$\beta_1$};
				\draw (270,63) node [anchor=north west][inner sep=0.75pt]  [font=\large] [align=left] {$\beta_{n_2}$};
				\draw (343,67) node [anchor=north west][inner sep=0.75pt]  [font=\large] [align=left] {$\alpha_{2}$};
				\draw (360,123) node [anchor=north west][inner sep=0.75pt]  [font=\large] [align=left] {$\alpha_{1}$};
				\draw (438,93) node [anchor=north west][inner sep=0.75pt]  [font=\large] [align=left] {$\alpha_{n_1}$};
				\draw (421,192) node [anchor=north west][inner sep=0.75pt]  [font=\large] [align=left] {$\gamma_1$};

			\end{tikzpicture}
			\caption{The tilting complex for the first Kauer move.}
			\label{fig:first-tilting-complex}	
		\end{figure}

		Then consider the complex $T_s=(P_1\oplus P_2\rightarrow P_s)$ with nonzero terms in degrees zero and one, and the differential is given by $(\alpha_{1},\gamma_1)$ (see Figure \ref{fig:first-tilting-complex}). Denote by $T=T_s\oplus(\bigoplus_{i=1}^{n}P_i)$ where $\{P_1,\cdots, P_{n}, P_s\}$ is a complete set of representatives of the isomorphism classes of projective $A$-modules.  By Theorem \ref{Okuyama}, $T$ is a tilting complex for $A$. In fact, certain modifications occur in the quiver $Q'$ of $B':=(\mathrm{End}_{\mathscr{K}^b(A)}(T))^{op}$. As an illustration, we describe explicit changes to the arrows around the vertices $v_1$ and $v_2$ in the ribbon graph.
		\begin{itemize}
			\item There is no arrow from $R_{n_1-1}$ to $T_s$, since the chain map $((\alpha_{n_1},0),0):\; T_s\rightarrow R_{n_1-1}$ is null-homotopic in $\mathscr{K}^b(A)$. At the same time, the chain map $\alpha_{1}\alpha_{n_1}:\;P_1\rightarrow R_{n_1-1}$ gives an arrow from $R_{n_1-1}$ to $P_1$, since it cannot be factor through $\mathrm{add}(T/(P_1\oplus R_{n_1-1}))$.
			
			\item There is an arrow form $P_1$ to $T_s$ which is given by $((\id,0),0):\;T_s\rightarrow P_1$, and there is an arrow from $T_s$ to $Q_2$ which is given by $((\beta_1,0)^t,0):\;Q_2\rightarrow T_s$.
		\end{itemize}
		The same modifications also change the arrows around the vertices $v_3$ and $v_4$. Therefore, the quiver of $B$ and the quiver $Q'$ of $B'$ are the same one. Indeed, the special cycle (under cyclic permutation) of $v_2$ is given by the following maps.
\[\begin{tikzcd}
	{P_1} & {Q_{n_2}} & \cdots & {Q_2} & {P_1\oplus P_2} & {P_1} \\
	&&&& {P_s}
	\arrow["{\beta_{n_2}}", from=1-1, to=1-2]
	\arrow["{\beta_{n_2-1}}", from=1-2, to=1-3]
	\arrow["{\beta_2}", from=1-3, to=1-4]
	\arrow["{(\beta_1,0)^t}", from=1-4, to=1-5]
	\arrow["{(\id,0)}", from=1-5, to=1-6]
	\arrow["{(\alpha_1,\gamma_1)}", from=1-5, to=2-5]
\end{tikzcd}\]
		Note that $((\id,0)(\beta_1,0)^t\beta_2\cdots\beta_{n_2})^{m(v_2)}(\id,0)$ is null-homotopic in $\mathscr{K}^b(A)$. Moreover, we can verify the relations in $B$ are also zero in $B'$. Thus there is an epimorphism form $B$ to $B'$.
		
		We now show $\dim_k(B)=\dim_k(B')$. Denote by $e_0, e_1,\cdots, e_n$ the idempotent elements in $B$ corresponding to edges in sf Brauer graph with $e_0$ corresponding to $s'$. Moreover, the set of vertices in $Q'$ and $\{T_0, T_1,\cdots, T_{n}\}$ with $T_0=T_s$ and $T_i=P_i$, $i=1,\cdots n$ have an one-to-one correspondence. We just need to show $\dim_k(e_i B e_j)=\dim_k(\ho_{\mathscr{K}^b(A)}(T_i,T_j))$. We only give a verification with $e_i=e_0$, $T_i=T_0$. Other calculations are almost same as these cases.

		\textit{Case 1. } If $e_j=e_0$, $T_j=T_0$. Then $\dim_k(e_0Be_0)=2+(m(v_2)-1)+(m(v_4)-1)=m(v_2)+m(v_4)$. Moreover, we have 
					$$
			\begin{array}{*{3}{lll}}
				\dim_k(\ho_{\mathscr{K}^b(A)}(T_0,T_0)) & =&\dim_k(\mathrm{End}_A(P_1\oplus P_2))+\dim_k(\mathrm{End}_A(P_s))\\&&-\dim_k(\ho_A(P_1\oplus P_2,P_s))-\dim_k(\ho_A(P_s,P_1\oplus P_2))\\
				& =& (m(v_1)+m(v_2))+(m(v_3)+m(v_4))+(m(v_1)+m(v_3))-2m(v_1)-2m(v_3)\\
				&=&m(v_2)+m(v_4).
			\end{array}
			$$
		
		\textit{Case 2. } If $e_j$ (resp. $T_j$) is corresponding an edge which is not involved in a same special cycle with $e_0$ (resp. $T_0$), then we have  $\dim_k(e_0 B e_j)=\dim_k(\ho_{\mathscr{K}^b(A)}(T_0,T_j))=0$.
		
		\textit{Case 3. } If $e_j$ (resp. $T_j$) is corresponding an edge which is in a same special cycle with $e_0$ (resp. $T_0$), without loss of generality, we can assume this special cycle is the cycle around the vertex $v_2$. Then $\dim_k(e_0Be_j)=m(v_2)$. Moreover, we have 
		$$
		\begin{array}{*{3}{lll}}
			\dim_k(\ho_{\mathscr{K}^b(A)}(T_0,T_j)) & =&\dim_k(\ho_A(P_1\oplus P_2,P_j))-\dim_k(\ho_A(P_s,P_j))\\
			& =& m(v_2).
		\end{array}
		$$
		
		Therefore, $\dim_k(e_0 B e_j)=\dim_k(\ho_{\mathscr{K}^b(A)}(T_0,T_j))$. 
		
		We also need to point out that there is something different from the proof in Brauer graph case. For example, consider two edges $e_{i_1}$ and $e_{i_2}$ in a same special cycle around the vertex $w_1$ in the sf Brauer graph algebra. Moreover, we assume all arrows which compose a subpath of this cycle from $e_{i_1}$ to $e_{i_2}$ are starting from labeled edges. In this case, if $e_{i_1}$ connect with vertices $w_1$ and $w_2$ in the sf Brauer graph, so does $e_{i_2}$. Then $\dim_k(e_{i_2}Be_{i_1})=m(w_1)+m(w_2)-1$. Moreover, consider the composition series of $P_{i_1}$, we have	$$P_{i_1}=\begin{array}{*{3}{lll}}
			&i_1&\\&\vdots&\\&i_2&\\\vdots&&\vdots\\&j_0&\\&\vdots&\\&i_1&
		\end{array}$$ where $j_0$ is corresponding to the labeled edge $\bar h$ such that $\overline{\rho^{i_1-j_0}(h)}$ corresponds to the idempotent $e_{i_1}$. Therefore, 
		$\dim_k(\ho_{\mathscr{K}^b(A)}(T_{i_2},T_{i_1}))=\dim_k(\ho_A(P_{i_2},P_{i_1}))=m(w_1)+m(w_2)-1.$
		
		In conclusion, we have $\dim_k(B)=\dim_k(B')$. Since there is an epimorphism form $B$ to $B'$, we finally show that $B\cong B'=(\mathrm{End}_{\mathscr{K}^b(A)}(T))^{op}$
	\end{proof}
		
	Finally, we end this section with some examples. 
	
	\begin{example}
		Consider the following sf Brauer graphs with multiplicity $1$ and all orientations are clockwise.
		
		
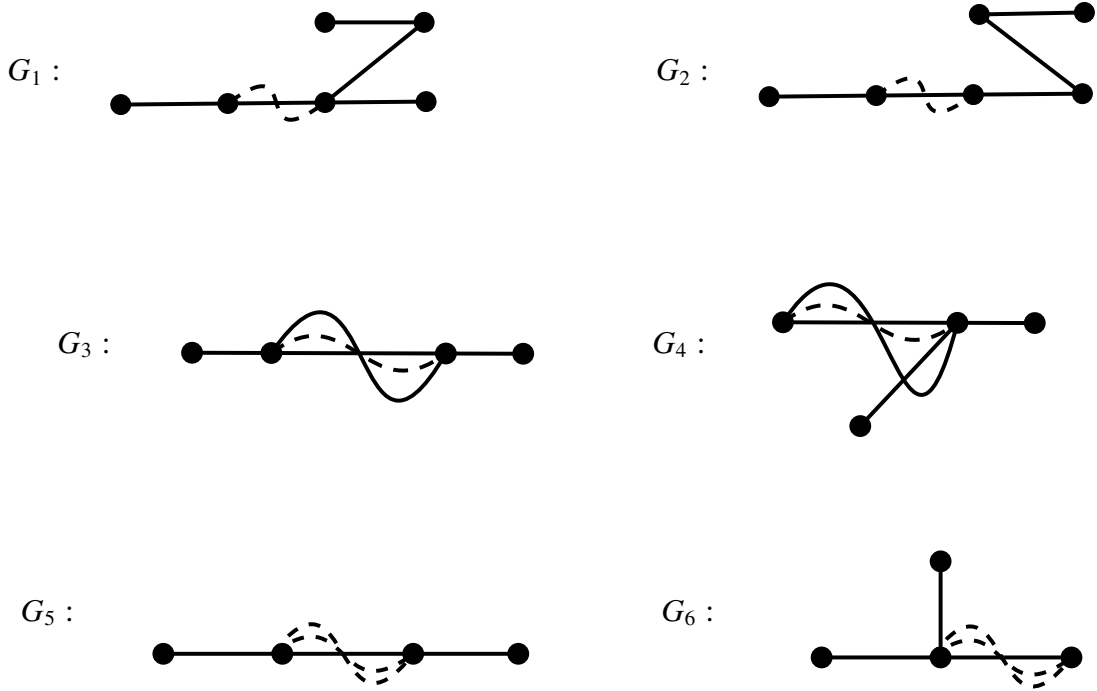
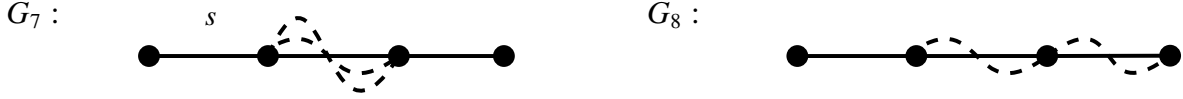
\begin{center}   
		
		\begin{tikzpicture}[x=0.75pt,y=0.75pt,yscale=-1,xscale=1]
			
			\draw  [fill={rgb, 255:red, 0; green, 0; blue, 0 }  ,fill opacity=1 ] (96,109) .. controls (96,106.24) and (98.24,104) .. (101,104) .. controls (103.76,104) and (106,106.24) .. (106,109) .. controls (106,111.76) and (103.76,114) .. (101,114) .. controls (98.24,114) and (96,111.76) .. (96,109) -- cycle ;
			\draw [line width=1.5]    (101,109) -- (155,108.45) -- (204,107.96) -- (254,67.5) -- (204,67.5) ;
			\draw [line width=1.5]    (204,107.96) -- (255,107.5) ;
			\draw  [fill={rgb, 255:red, 0; green, 0; blue, 0 }  ,fill opacity=1 ] (150,108.45) .. controls (150,105.69) and (152.24,103.45) .. (155,103.45) .. controls (157.76,103.45) and (160,105.69) .. (160,108.45) .. controls (160,111.21) and (157.76,113.45) .. (155,113.45) .. controls (152.24,113.45) and (150,111.21) .. (150,108.45) -- cycle ;
			\draw  [fill={rgb, 255:red, 0; green, 0; blue, 0 }  ,fill opacity=1 ] (199,107.96) .. controls (199,105.19) and (201.24,102.96) .. (204,102.96) .. controls (206.77,102.96) and (209,105.19) .. (209,107.96) .. controls (209,110.72) and (206.77,112.96) .. (204,112.96) .. controls (201.24,112.96) and (199,110.72) .. (199,107.96) -- cycle ;
			\draw  [fill={rgb, 255:red, 0; green, 0; blue, 0 }  ,fill opacity=1 ] (199,67.5) .. controls (199,64.74) and (201.24,62.5) .. (204,62.5) .. controls (206.76,62.5) and (209,64.74) .. (209,67.5) .. controls (209,70.26) and (206.76,72.5) .. (204,72.5) .. controls (201.24,72.5) and (199,70.26) .. (199,67.5) -- cycle ;
			\draw  [fill={rgb, 255:red, 0; green, 0; blue, 0 }  ,fill opacity=1 ] (250,107.5) .. controls (250,104.74) and (252.24,102.5) .. (255,102.5) .. controls (257.76,102.5) and (260,104.74) .. (260,107.5) .. controls (260,110.26) and (257.76,112.5) .. (255,112.5) .. controls (252.24,112.5) and (250,110.26) .. (250,107.5) -- cycle ;
			\draw  [fill={rgb, 255:red, 0; green, 0; blue, 0 }  ,fill opacity=1 ] (249,67.5) .. controls (249,64.74) and (251.24,62.5) .. (254,62.5) .. controls (256.76,62.5) and (259,64.74) .. (259,67.5) .. controls (259,70.26) and (256.76,72.5) .. (254,72.5) .. controls (251.24,72.5) and (249,70.26) .. (249,67.5) -- cycle ;
			\draw  [fill={rgb, 255:red, 0; green, 0; blue, 0 }  ,fill opacity=1 ] (423,105) .. controls (423,102.24) and (425.24,100) .. (428,100) .. controls (430.76,100) and (433,102.24) .. (433,105) .. controls (433,107.76) and (430.76,110) .. (428,110) .. controls (425.24,110) and (423,107.76) .. (423,105) -- cycle ;
			\draw [line width=1.5]    (428,105) -- (482,104.45) -- (531,103.96) -- (586,103.5) -- (534,63.5) -- (587,62.5) ;
			\draw  [fill={rgb, 255:red, 0; green, 0; blue, 0 }  ,fill opacity=1 ] (477,104.45) .. controls (477,101.69) and (479.24,99.45) .. (482,99.45) .. controls (484.76,99.45) and (487,101.69) .. (487,104.45) .. controls (487,107.21) and (484.76,109.45) .. (482,109.45) .. controls (479.24,109.45) and (477,107.21) .. (477,104.45) -- cycle ;
			\draw  [fill={rgb, 255:red, 0; green, 0; blue, 0 }  ,fill opacity=1 ] (526,103.96) .. controls (526,101.19) and (528.24,98.96) .. (531,98.96) .. controls (533.77,98.96) and (536,101.19) .. (536,103.96) .. controls (536,106.72) and (533.77,108.96) .. (531,108.96) .. controls (528.24,108.96) and (526,106.72) .. (526,103.96) -- cycle ;
			\draw  [fill={rgb, 255:red, 0; green, 0; blue, 0 }  ,fill opacity=1 ] (529,63.5) .. controls (529,60.74) and (531.24,58.5) .. (534,58.5) .. controls (536.76,58.5) and (539,60.74) .. (539,63.5) .. controls (539,66.26) and (536.76,68.5) .. (534,68.5) .. controls (531.24,68.5) and (529,66.26) .. (529,63.5) -- cycle ;
			\draw  [fill={rgb, 255:red, 0; green, 0; blue, 0 }  ,fill opacity=1 ] (582,62.5) .. controls (582,59.74) and (584.24,57.5) .. (587,57.5) .. controls (589.76,57.5) and (592,59.74) .. (592,62.5) .. controls (592,65.26) and (589.76,67.5) .. (587,67.5) .. controls (584.24,67.5) and (582,65.26) .. (582,62.5) -- cycle ;
			\draw  [fill={rgb, 255:red, 0; green, 0; blue, 0 }  ,fill opacity=1 ] (581,103.5) .. controls (581,100.74) and (583.24,98.5) .. (586,98.5) .. controls (588.76,98.5) and (591,100.74) .. (591,103.5) .. controls (591,106.26) and (588.76,108.5) .. (586,108.5) .. controls (583.24,108.5) and (581,106.26) .. (581,103.5) -- cycle ;
			\draw [line width=1.5]  [dash pattern={on 5.63pt off 4.5pt}]  (155,108.45) .. controls (195,78.45) and (164,137.96) .. (204,107.96) ;
			\draw [line width=1.5]  [dash pattern={on 4.5pt off 4.5pt}]  (482,104.45) .. controls (522,74.45) and (491,133.96) .. (531,103.96) ;
			
			\draw (44,85) node [anchor=north west][inner sep=0.75pt]  [font=\large] [align=left] {$\Gamma^1_L:$};
			\draw (360,85) node [anchor=north west][inner sep=0.75pt]  [font=\large] [align=left] {$\Gamma^2_L:$};

		\end{tikzpicture}

\tikzset{every picture/.style={line width=0.75pt}} 

\begin{tikzpicture}[x=0.75pt,y=0.75pt,yscale=-1,xscale=1]
	
	\draw  [fill={rgb, 255:red, 0; green, 0; blue, 0 }  ,fill opacity=1 ] (127,121) .. controls (127,118.24) and (129.24,116) .. (132,116) .. controls (134.76,116) and (137,118.24) .. (137,121) .. controls (137,123.76) and (134.76,126) .. (132,126) .. controls (129.24,126) and (127,123.76) .. (127,121) -- cycle ;
	\draw  [fill={rgb, 255:red, 0; green, 0; blue, 0 }  ,fill opacity=1 ] (277,121) .. controls (277,118.24) and (279.24,116) .. (282,116) .. controls (284.76,116) and (287,118.24) .. (287,121) .. controls (287,123.76) and (284.76,126) .. (282,126) .. controls (279.24,126) and (277,123.76) .. (277,121) -- cycle ;
	\draw [line width=1.5]  [dash pattern={on 5.63pt off 4.5pt}]  (132,121) .. controls (94,68.5) and (156,65.5) .. (282,121) ;
	\draw [line width=1.5]    (132,121) .. controls (250,67.5) and (357,69.5) .. (282,121) ;
	\draw [line width=1.5]    (132,121) -- (212,121) ;
	\draw  [fill={rgb, 255:red, 0; green, 0; blue, 0 }  ,fill opacity=1 ] (207,121) .. controls (207,118.24) and (209.24,116) .. (212,116) .. controls (214.76,116) and (217,118.24) .. (217,121) .. controls (217,123.76) and (214.76,126) .. (212,126) .. controls (209.24,126) and (207,123.76) .. (207,121) -- cycle ;
	\draw [line width=1.5]    (132,121) .. controls (172,153.5) and (243,154.5) .. (282,121) ;
	\draw  [fill={rgb, 255:red, 0; green, 0; blue, 0 }  ,fill opacity=1 ] (458,121.01) .. controls (458,118.25) and (460.24,116.01) .. (463,116.01) .. controls (465.76,116.01) and (468,118.25) .. (468,121.01) .. controls (468,123.77) and (465.76,126.01) .. (463,126.01) .. controls (460.24,126.01) and (458,123.77) .. (458,121.01) -- cycle ;
	\draw  [fill={rgb, 255:red, 0; green, 0; blue, 0 }  ,fill opacity=1 ] (608,121.01) .. controls (608,118.25) and (610.24,116.01) .. (613,116.01) .. controls (615.76,116.01) and (618,118.25) .. (618,121.01) .. controls (618,123.77) and (615.76,126.01) .. (613,126.01) .. controls (610.24,126.01) and (608,123.77) .. (608,121.01) -- cycle ;
	\draw [line width=1.5]  [dash pattern={on 5.63pt off 4.5pt}]  (463,121.01) .. controls (425,68.51) and (487,65.51) .. (613,121.01) ;
	\draw [line width=1.5]    (463,121.01) .. controls (581,67.51) and (688,69.51) .. (613,121.01) ;
	\draw [line width=1.5]    (533,121.01) -- (613,121.01) ;
	\draw  [fill={rgb, 255:red, 0; green, 0; blue, 0 }  ,fill opacity=1 ] (528,121.01) .. controls (528,118.25) and (530.24,116.01) .. (533,116.01) .. controls (535.76,116.01) and (538,118.25) .. (538,121.01) .. controls (538,123.77) and (535.76,126.01) .. (533,126.01) .. controls (530.24,126.01) and (528,123.77) .. (528,121.01) -- cycle ;
	\draw [line width=1.5]    (463,121.01) .. controls (503,153.51) and (574,154.51) .. (613,121.01) ;
	
	\draw (56,102) node [anchor=north west][inner sep=0.75pt]  [font=\large] [align=left] {$\Gamma^3_L:$};
	\draw (399,102) node [anchor=north west][inner sep=0.75pt]  [font=\large] [align=left] {$\Gamma^4_L:$};

\end{tikzpicture}

\tikzset{every picture/.style={line width=0.75pt}} 

\begin{tikzpicture}[x=0.75pt,y=0.75pt,yscale=-1,xscale=1]
	
	\draw  [fill={rgb, 255:red, 0; green, 0; blue, 0 }  ,fill opacity=1 ] (127,121) .. controls (127,118.24) and (129.24,116) .. (132,116) .. controls (134.76,116) and (137,118.24) .. (137,121) .. controls (137,123.76) and (134.76,126) .. (132,126) .. controls (129.24,126) and (127,123.76) .. (127,121) -- cycle ;
	\draw  [fill={rgb, 255:red, 0; green, 0; blue, 0 }  ,fill opacity=1 ] (277,121) .. controls (277,118.24) and (279.24,116) .. (282,116) .. controls (284.76,116) and (287,118.24) .. (287,121) .. controls (287,123.76) and (284.76,126) .. (282,126) .. controls (279.24,126) and (277,123.76) .. (277,121) -- cycle ;
	\draw [line width=1.5]  [dash pattern={on 5.63pt off 4.5pt}]  (132,121) .. controls (94,68.5) and (156,65.5) .. (282,121) ;
	\draw [line width=1.5]  [dash pattern={on 5.63pt off 4.5pt}]  (132,121) .. controls (250,67.5) and (357,69.5) .. (282,121) ;
	\draw [line width=1.5]    (82,121) -- (132,121) ;
	\draw  [fill={rgb, 255:red, 0; green, 0; blue, 0 }  ,fill opacity=1 ] (77,121) .. controls (77,118.24) and (79.24,116) .. (82,116) .. controls (84.76,116) and (87,118.24) .. (87,121) .. controls (87,123.76) and (84.76,126) .. (82,126) .. controls (79.24,126) and (77,123.76) .. (77,121) -- cycle ;
	\draw [line width=1.5]    (132,121) .. controls (172,153.5) and (243,154.5) .. (282,121) ;
	\draw  [fill={rgb, 255:red, 0; green, 0; blue, 0 }  ,fill opacity=1 ] (437,121.01) .. controls (437,118.25) and (439.24,116.01) .. (442,116.01) .. controls (444.76,116.01) and (447,118.25) .. (447,121.01) .. controls (447,123.77) and (444.76,126.01) .. (442,126.01) .. controls (439.24,126.01) and (437,123.77) .. (437,121.01) -- cycle ;
	\draw  [fill={rgb, 255:red, 0; green, 0; blue, 0 }  ,fill opacity=1 ] (587,121.01) .. controls (587,118.25) and (589.24,116.01) .. (592,116.01) .. controls (594.76,116.01) and (597,118.25) .. (597,121.01) .. controls (597,123.77) and (594.76,126.01) .. (592,126.01) .. controls (589.24,126.01) and (587,123.77) .. (587,121.01) -- cycle ;
	\draw [line width=1.5]  [dash pattern={on 5.63pt off 4.5pt}]  (442,121.01) .. controls (404,68.51) and (466,65.51) .. (592,121.01) ;
	\draw [line width=1.5]    (442,121.01) .. controls (560,67.51) and (667,69.51) .. (592,121.01) ;
	\draw  [fill={rgb, 255:red, 0; green, 0; blue, 0 }  ,fill opacity=1 ] (537,121.01) .. controls (537,118.25) and (539.24,116.01) .. (542,116.01) .. controls (544.76,116.01) and (547,118.25) .. (547,121.01) .. controls (547,123.77) and (544.76,126.01) .. (542,126.01) .. controls (539.24,126.01) and (537,123.77) .. (537,121.01) -- cycle ;
	\draw  [fill={rgb, 255:red, 0; green, 0; blue, 0 }  ,fill opacity=1 ] (227,121) .. controls (227,118.24) and (229.24,116) .. (232,116) .. controls (234.76,116) and (237,118.24) .. (237,121) .. controls (237,123.76) and (234.76,126) .. (232,126) .. controls (229.24,126) and (227,123.76) .. (227,121) -- cycle ;
	\draw  [fill={rgb, 255:red, 0; green, 0; blue, 0 }  ,fill opacity=1 ] (387,121.01) .. controls (387,118.25) and (389.24,116.01) .. (392,116.01) .. controls (394.76,116.01) and (397,118.25) .. (397,121.01) .. controls (397,123.77) and (394.76,126.01) .. (392,126.01) .. controls (389.24,126.01) and (387,123.77) .. (387,121.01) -- cycle ;
	\draw  [fill={rgb, 255:red, 0; green, 0; blue, 0 }  ,fill opacity=1 ] (637,121.01) .. controls (637,118.25) and (639.24,116.01) .. (642,116.01) .. controls (644.76,116.01) and (647,118.25) .. (647,121.01) .. controls (647,123.77) and (644.76,126.01) .. (642,126.01) .. controls (639.24,126.01) and (637,123.77) .. (637,121.01) -- cycle ;
	\draw [line width=1.5]    (232,121) -- (282,121) ;
	\draw [line width=1.5]    (392,121.01) -- (442,121.01) ;
	\draw [line width=1.5]    (542,121.01) -- (592,121.01) ;
	\draw [line width=1.5]    (592,121.01) -- (642,121.01) ;
	
	\draw (41,92) node [anchor=north west][inner sep=0.75pt]  [font=\large] [align=left] {$\Gamma^5_L:$};
	\draw (350,91) node [anchor=north west][inner sep=0.75pt]  [font=\large] [align=left] {$\Gamma^6_L:$};
	\draw (98,108) node [anchor=north west][inner sep=0.75pt]   [align=left] {$s$};
\end{tikzpicture}	
		\end{center}
		Denote  the sf Brauer graph algebra associated with $\Gamma^i_L$ by $A_i$, $i=\{1,\cdots, 6\}$.
		By the Kauer move of the first type, we have the sf Brauer graph algebras $A_1$ and $A_2$ are derived equivalent. By the Kauer move of the second type, we have the sf Brauer graph algebras $A_3$ and $A_4$ are derived equivalent. 
		
		Moreover, not all derived equivalences between sf Brauer graph algebras can be obtained from Kauer moves. For example,  the sf Brauer graph algebras $A_5$ and $A_6$ are derived equivalent by choosing the Okuyama tilting complex with $I=\{S\}$, the simple module corresponding to the edge $s$ in $\Gamma^5_L$.
	\end{example}

	\section{Symmetric special quasi-biserial algebras}\label{sec:sym-sqb-4}
	
	We now begin to prove that all symmetric special quasi-biserial algebras are sf Brauer graph algebras. In this section, we assume that the algebras discussed are basic and indecomposable. Without loss of generality, we assume each sf Brauer graph algebra $A$ discussed in this section is indecomposable and contains a non-single arrow in its quiver, otherwise, $A$ is a Nakayama algebra which is contained in Brauer tree algebras.
	
	Firstly, we recall the concept of an algebra with arrow-free socle. Let $A=\qa$ be an indecomposable finite-dimensional algebra with $I$ an admissible ideal in the path algebra $kQ$.
	
	\begin{definition}\textnormal{(\cite[Definition 3.1]{GS1})}
		We say that the socle of $A$ is {\it arrow-free} if, for each $\alpha\in Q_1$, we have $\alpha\notin\soc(_AA)$ and $\alpha\notin\soc(A_A)$.
	\end{definition}
	
	\begin{proposition}\textnormal{(\cite[Proposition 3.2]{GS1})}\label{prop:self-socfree}
		Let $A$ be self-injective and $\rad^2(A)\neq 0$. Then the socle of $A$ is arrow-free.
	\end{proposition}
	
	\begin{proposition}\textnormal{(\cite[Lemma 3.3]{GS1})}\label{soc-free}
		If the socle of $A$ is arrow-free then for each arrow $\alpha$ in $Q$, there are arrows $\beta$ and $\gamma$ such that $\alpha\beta\notin I$ and $\gamma\alpha\notin I$.
	\end{proposition}
	
	Now let $A=\qa$ be a symmetric special quasi-biserial algebra. Since $A$ is symmetric, $A$ is self-injective. if $\rad^2(A)=0$, then $A$ is simply a symmetric Nakayama algebra, which is just a Brauer tree algebra. Therefore, in following discussion, we assume $\rad^2(A)\neq 0$. Then by Proposition \ref{prop:self-socfree}, the socle of $A$ is arrow-free. 
	
	We call  a cycle $C=\alpha_n\cdots\alpha_1$ in $Q$ is {\it basic} if 
	
	\begin{itemize}
		\item $\alpha_{n-1}\cdots\alpha_{1},\;\alpha_1\alpha_n\cdots\alpha_{2}\notin I$;
		\item $C$ cannot be written as the form $C=C'^n$ with $C'$ another cycle in $Q$ and $n> 1$.
	\end{itemize}
	 We call a set $\{C_1,\cdots,C_r\}$ of basic cycles is {\it special} if
	
	\begin{itemize}
		\item for each arrow $\alpha$ in $Q$, $\alpha$ must occur in some $C_i$. Moreover, if $\alpha$ is a non-single arrow, then $\alpha$ occurs in exactly one $C_i$.
	\end{itemize}
	Now we can define the basic cycles and a special set of basic cycles in $A$.
	
	By Proposition \ref{soc-free} and Definition \ref{sqb}, for each non-single arrow $\alpha$, there are exactly one arrow $\beta$ and exactly one arrow $\gamma$, such that $\alpha\beta\notin I$ and $\gamma\alpha\notin I$. Since $A$ is finite dimensional, there exists a maximal path $C_\alpha'=c_1c_2$ contains $\alpha$  which is nonzero in $\soc(A)$. 
	If $C_\alpha'$ is not a cycle, then $c_2c_1=0$ contradict with the symmetry of $A$. Let $C_\alpha$ be the basic cycle of $C_\alpha'$.
	Since $A=kQ/I$ is indecomposable and $Q$ contains at least one non-single arrow, there is no basic cycles consisting of all single arrows in $A$. 
	
	Consider the set of all cycles $\mathcal{S}$ in $Q$. Define the cyclic permutation on $\mathcal{S}$ which is given by $\tau:\mathcal{S}\rightarrow\mathcal{S}$, $\alpha_{n}\cdots\alpha_{1}\mapsto \alpha_{1}\alpha_{n}\cdots\alpha_{2}$. Therefore, under cyclic permutation, we find a unique special set of basic cycles $\{C_1,\cdots,C_n\}$ in $A$.
	
	\begin{lemma}\label{cycle}
		Let $A=\qa$ be a symmetric special quasi-biserial algebra and $C$ be a basic cycle of $A$ defined above. $\psi$ is the symmetric $k$-linear form of $A$. The following properties hold.
		\begin{enumerate}
			\item There is an integer $\m(C)>0$ such that $C^{\m(C)}$ is a nonzero element in $\soc(A)$.
			
			\item We have $\m(C)=\m(\tau^i(C))$ and $(\tau^i(C))^{\m(\tau^i(C))}$ is a nonzero element in $\soc(A)$, for $0\leq i\leq \ell(C)-1$.
			
			\item We have $\psi(C^{\m(C)}) =\psi(\tau^i(C)^{\m(\tau^i(C))})$, for $0\leq i\leq \ell(C)-1$.
		\end{enumerate}
	\end{lemma}

	\begin{proof}
		Part (1) is directly follows from the definition of the basic cycles in $A$.
		
		Suppose that $C^{\m(C)}$ is a nonzero element in $\soc(A)$. It is suffices to show $\tau(C)^{\m(C)}\in\soc(A)$. First note that, using $\psi(xy)=\psi(yx)$, for any $x,y\in A$, we see that  $\psi(C^{\m(C)}) =\psi(\tau(C)^{\m(C)})$. Hence $\tau(C)^{\m(C)}$ is not zero in $A$. Suppose for contradiction that $\tau(C)^{\m(C)}\notin\soc(A)$, then there exist an arrow $\alpha$ such that $\tau(C)^{\m(C)}\alpha\neq 0$. However, $\tau(C)^{\m(C)}\alpha=\alpha C^{\m(C)}$ and hence $\tau(C)^{\m(C)}\alpha=0$ in $A$ since $C^{\m(C)}\in\soc(A)$, a contradiction.
		
		Part (3) follows since $\psi(xy)=\psi(yx)$ and there exist paths $p_1,p_2$ in $Q$ such that $p_1p_2=C^{\m(C)}$ and $p_2p_1=\tau^i(C)^{\m(\tau^i(C))}$, for $0\leq i\leq \ell(C)-1$.
	\end{proof}
	
	We give a generalized version of \cite[Lemma 4.10]{GS1} since the proof of this lemma did not use the condition that $A$ is special multiserial.
	
	\begin{lemma}\label{p=p'}
		Let $A=\qa$ be an indecomposable symmetric $k$-algebra and let $\psi$ be the non-degenerate symmetric $k$-linear form of $A$. Let $e$ be a primitive idempotent in $A$ and let $p$ and $p'$ be nonzero elements in $e\soc(A)e$ such that $\psi(p)=\psi(p')$. Then $p=p'$.
	\end{lemma}
	
	For the next result we need to assume that the field $k$  is algebraically closed. Keeping the notation above, we have the following.
	
	\begin{proposition}\label{prop:quiver-relations}
		Let $k$ be an algebraically closed field, let $A$ be a symmetric special quasi-biserial $k$-algebra, and let $Q$ be the quiver of $A$. Then there exist a surjection $\pi:kQ\rightarrow A$ such that
		\begin{enumerate}
			\item $\ker\pi$ is admissible, and
			
			\item if $C_1$ and $C_2$ are basic cycle starting at a same vertex $v$ in $Q$, then $\pi(C_1^{\m(C_1)})=\pi(C_2^{\m(C_2)})$.
		\end{enumerate}
	\end{proposition}

	\begin{proof}
		Since $A$ is assumed to be finite-dimensional and basic, there exist a surjection $\pi':kQ\rightarrow A$ such that $\ker(\pi')$ is admissible. Let $\psi$ be the non-degenerate symmetric $k$-linear form of $A$. We now construct a surjection $\pi:kQ\rightarrow A$ by defining, for each arrow $\alpha$ in $Q$, a nonzero constant $\lambda_\alpha\in k$ such that by setting $\pi(\alpha)=\lambda_\alpha(\pi'(\alpha))$ the desired properties hold. Since $\ker(\pi')$ is admissible, clearly $\ker(\pi)$ is admissible.
		
		Without loss of generality, we assume each basic cycle in special set of $A$ contains a non-single arrow $\alpha$ and denote this basic cycle by $C_\alpha$ and $\m(C_\alpha)$ by $\m_\alpha$. By Lemma \ref{cycle}, $C_\alpha^{\m_\alpha}\in e_v\soc(A)e_v$ where $v=s(C_\alpha)$ and $e_v$ is the associated primitive idempotent in $A$. We know that $\psi(C_\alpha^{\m_\alpha})\neq 0$ in $A$. Let $\lambda_\alpha=(\psi(C_\alpha^{\m_\alpha}))^{-1/\m_\alpha}$, and for arrow $\beta\mid C_\alpha$ with $\beta\neq\alpha$, set $\lambda_\beta=1$.
		Then we have that for each basic cycle $C$ in $A$, $\psi(C^{\m(C)})=1$. Applying Lemma \ref{p=p'}, we have $\pi(C_1^{\m(C_1)})=\pi(C_2^{\m(C_2)})$ for each basic cycle $C_1$ and $C_2$ in the assumption.
	\end{proof}
	
	Therefore, let $A=\qa$ be a symmetric special quasi-biserial algebra with $k$ is algebraically closed. By Proposition \ref{prop:quiver-relations}, we can regard every path $p$ in $Q$ as an element in $A$ under the given surjection $\pi:kQ\rightarrow A$. Without loss of generality, we can assume that a minimal set of relations $\rho$ generating $I$ contains only zero relations and commutativity relations of the form $p-q$ for $p,q$ paths in $Q$ such that $p,q\notin\rho$. Otherwise, there exists a relation of the form $p-k_0q$ with $p\neq q$ in $\rho$, then we can extend $p$ and $q$ to some cycles $C_1^{\m(C_1)}=pp'$ and $C_2^{\m(C_2)}=qq'$, respectively. However, since $C_1^{\m(C_1)}=C_2^{\m(C_2)}$ in $\soc(A)$ and $A$ is special quasi-biserial, we have $p'=q'$ and $k_0=1$.
	
	Let $i$ be a vertex in $Q$. If the  indecomposable projective $A$-module $P_i$ at $i$ is uniserial, then there exists a unique non-trivial maximal cycle $p$ in $(Q,I)$ with $s(p)=t(p)=i$. Then we write $P_i=P_i(p,e_i)=P_i(p)$ or $P(p)$ for short. If $P_i$ is quasi-biserial, then there exist two distinct non-trivial maximal cycle $p,q$ in $Q$ with $s(p)=s(q)=t(p)=t(q)=i$ such that $p-q\in I$. Denote the maximal paths by $r_1,r_2$ such that $p-q=r_1(p_1-q_1)r_2$. Then $p_1-p_2\in I$. Otherwise, the ideal $I'=\langle p_1-p_2\rangle$ is a nonzero ideal such that for the symmetric bilinear form $\psi$ of $A$, $\psi(I)=0$, contradict with $A$ symmetric. At that time, we write $P_i=P_i(p,q)$ or $P(p,q)$ for short. Since $A$ is symmetric, the indecomposable projective module at $i$ is also the indecomposable injective module at $i$.
	
	\begin{lemma}\label{k-basis of proj}
		Let $A=\qa_A$ and $B=\qa_B$ be symmetric special quasi-biserial algebras. Suppose that for every vertex in $Q$, the projective indecomposable modules of $A$ and $B$ have $k$-basis given by the same paths in $Q$. Then the algebras $A$ and $B$ are isomorphic.
	\end{lemma}
	
	\begin{proof}
		We prove this lemma by showing that $I_A=I_B$. Denote the indecomposable projective $A$-module (resp. $B$-module) at vertex $i$ by $P_i^A$ (resp. $P_i^B$). To complete the proof, it suffices to demonstrate that every generating relation of $I_A$ must also be a generating relation of $I_B$.
		
		Suppose $\alpha_{n}\cdots\alpha_{1}$ is a path in $Q$ with $\alpha_{n}\cdots\alpha_{1}\in I_A$. Now assume that $\alpha_{n}\cdots\alpha_{1}\notin I_B$. Then there is a path $p_1$ in $Q$ such that $p=p_1\alpha_{n}\cdots\alpha_{1}$ and $P_i^B=P(p,q)$ for some possibly trivial path $q$ with $i=s(\alpha_{1})$. But $p\in I_A$ since $\alpha_{n}\cdots\alpha_{1}\in I_A$. Therefore, $P_i^A\not\cong P(p,q)$, a contradiction and thus $\alpha_{n}\cdots\alpha_{1}\in I_B$. By symmetry of the argument this implies that $\alpha_{n}\cdots\alpha_{1}\in I_A$ if and only if $\alpha_{n}\cdots\alpha_{1}\in I_B$. Now suppose that $p,q$ are distinct parallel paths in $Q$ with $p,q\notin I_A$, $s(p)=s(q)=i$ and $0\neq p-q\in I_A$. Then there exists a path $p_0$ in $Q$ with $t(p_0)=i$ such that $p_0p=p_0q$ are cycles in $Q$. Moreover, $P_i^A\cong P(p_0p,p_0q)$. If $p-q\notin I_B$, since $p,q\notin I_B$, the $k$-basis of $P_i^B$ contains $p$ and $q$ independently. At the same time, the $k$-basis of $P_i^A$ only contains $p$ since $p=q$ in $A$, a contradiction.	
	\end{proof}
	
	Finally, we prove the main theorem in this section.
	
	\begin{theorem}\label{ssqb is sf-BGA}
		Let $A=\qa$ be a finite-dimensional $k$-algebra with $k$ algebraically closed. Then $A$ is a symmetric special quasi-biserial algebra if and only if it is an sf Brauer graph algebra.
	\end{theorem}
	
	\begin{proof}
		In \cite[Section 2.3.1]{Sch}, the author have constructed a ribbon graph of a symmetric special biserial algebra. We generalized this construction to the case of symmetric special quasi-biserial cases.
		
		Without loss of generality, assume $\rad^2(A)\neq 0$. Then by Proposition \ref{soc-free}, the socle of $A$ is arrow-free, and thus there is a unique set $S=\{C_1,\cdots,C_r\}$ of basic cycles in $A$ under cyclic permutation. Now we construct a ribbon graph $\Gamma_A=(V_A,H_A,s,\iota,\rho)$ associated with $A$, a degree function $d$ and a quantized function $q$ of $A$.
		
		Let $H_A=\{h_1, h_{-1}\;|\; h\in Q_0\}$. Then for each $h_i\in H_A$ ($i=\pm 1$), define $\iota(h_i)=h_{-i}$. Denote by $V_A'$ the vertices in $\Gamma_A$ which are not truncated and let $V_A'$ coincide with the set $S$. Define $\m(v)=\m(C_i)$ for each $v\in V_A'$ corresponding to $C_i$ in $S$ ($\m$ is defined in Lemma \ref{cycle}).
		
		For each $h\in Q_0$, 
		\begin{itemize}
			\item if there is exactly one basic cycle $C_v$ (which is corresponding to $v\in V_A'$) in $S$ passing through $h$, then let $s(h_1)=v$; 
			\item if there are two basic cycles $C_v,C_w$ (which are corresponding to $v,w\in V_A'$) in $S$ passing through $h$, then let $s(h_1)=v$ and $s(h_{-1})=w$, respectively.
		\end{itemize} 
		Then denote by $H_A':=s^{-1}(V_A')$, which is a subset of $H_A$. For each $h_i,f_j\in H_A'$ with $s(h_i)=s(f_j)=v$ in $V_A'$, define $\rho(h_i)=f_j$ if there is an arrow $\alpha$ from $h$ to $f$ in $Q$ and $\alpha\;|\;C_v$, where $C_v$ is the special cycle corresponding to $v$. 
		Let $H:=H_A\backslash H_A'$. Define $V_A=V_A'\cup H$ and for each $h_{-1}\in H$, define $s(h_{-1})=h_{-1}$ in $V_A$. Moreover, let $\m(h_{-1})=1$ for each $h_{-1}\in V_A$.
		By Definition \ref{def:ribbon-graph}, $\Gamma_A=(V_A,H_A,s,\iota,\rho)$ is a ribbon graph.
		
		Now reconsider the special set $S$ of basic cycles. For each arrow $\alpha\in Q_1$, if there exist two distinct basic cycles $C_i, C_j\in S$ such that $\alpha\mid C_i$ and $\alpha\mid C_j$, then we label the edge in $\Gamma_A$ corresponding to the vertex $s(\alpha)$. The set $L$ of labeled edges in $\Gamma_A$ consists of all such edges defined in this way.
		Denote by $B$ the sf Brauer graph algebra associated with $((\Gamma_A)_L,\m)$. Since the subpaths of basic cycles in $A$ naturally form a $k$-basis of $A$, by Lemma \ref{k-basis of proj}, $A\cong B$ as $k$-algebras.
	\end{proof}
	
	Therefore, combining Theorem \ref{ssqba quotient} and Theorem \ref{ssqb is sf-BGA}, we get the following corollary.
	
	\begin{corollary}
		Let $A=\qa$ be a finite-dimensional $k$-algebra with $k$ algebraically closed. If $A$ is special quasi-biserial, then $A$ is a quotient of an sf Brauer graph algebra.
	\end{corollary}

\end{document}